\documentclass[11pt,reqno]{amsart}
\usepackage{amssymb,amsfonts,amsthm,amscd,stmaryrd,dsfont,esint,upgreek,constants,todonotes}
\usepackage[mathcal]{euscript}
\usepackage[latin1]{inputenc}   
\usepackage[mathcal]{euscript}
\usepackage{graphicx,color}
\numberwithin{equation}{section}
\newcommand{\N}{\mathbb N}

\newcommand{\R}{\mathbb R}
\def\E{\mathbb E}
\def\P{\mathbb P}
\newconstantfamily{C}{symbol=C}
\newconstantfamily{c}{symbol=c}
\newconstantfamily{O}{symbol=\Omega}

\def\XXint#1#2#3{{\setbox0=\hbox{$#1{#2#3}{\int}$}
\vcenter{\hbox{$#2#3$}}\kern-.5\wd0}}

\numberwithin{equation}{section}
\newtheorem{thm}{Theorem}[section]
\newtheorem{lem}[thm]{Lemma}
\newtheorem{cor}[thm]{Corollary}
\newtheorem{prop}[thm]{Proposition}

\theoremstyle{definition}
\newtheorem{defn}[thm]{Definition}
\newtheorem{rmk}[thm]{Remark}

\def\smallnegint{\mathop{\int\mkern-13mu
        \raise.5ex\hbox{${\scriptscriptstyle\diagup}$}}\nolimits}
\def\ds{\displaystyle}

\def\ep{\varepsilon}

\def\ssetminus{\,\raise.4ex\hbox{$\scriptstyle\setminus$}\,}

\newcommand{\be}{\begin{equation}}
\newcommand{\ee}{\end{equation}}

\renewcommand{\d}{d}
\newcommand{\Rd}{\mathbb{R}^\d}

\renewcommand{\bar}{\overline}
\renewcommand{\tilde}{\widetilde}
\renewcommand{\hat}{\widehat}

\begin{document}
\title{On first order Mean Field Game systems with a common noise}

\author{Pierre Cardaliaguet$^{1}$ and Panagiotis E Souganidis$^{2}$}

\dedicatory{Version: \today}


\maketitle      

\begin{abstract}
We consider Mean Field Games without idiosyncratic but with Brownian type common noise. We introduce a notion of solutions of the associated backward-forward system of stochastic partial differential equations. We show that the solution exists and is unique for  monotone coupling functions. This  the first general result for  solutions of the Mean Field Games  system with common and no idiosynctratic noise.
We also use the solution to find approximate optimal strategies (Nash equilibria)   for $N$-player differential games with common but no idiosyncratic noise. An important step in the analysis is the study of the well-posedness of a stochastic backward Hamilton-Jacobi equation.  
\end{abstract}

\section{Introduction}

We consider Mean Field Games (MFG for short) without idiosyncratic but with Brownian-type common noise described by the backward-forward system of stochastic partial differential equations (SPDEs for short)
\be
\label{e.MFGstoch}
\left\{ \begin{array}{l}
\ds 
 d_{t} u_{t} = \bigl[ - \beta \Delta u_{t} + H(Du_{t},x) - F(x,m_{t}) -  \sqrt{2\beta} {\rm div}(v_{t}) \bigr] dt\\[2mm]
\hskip2.5in \ds  + v_{t} \cdot \sqrt{2\beta}dW_{t}
\ \  {\rm in } \ \ \R^d\times (0,T), \\[2mm]
\ds d_{t} m_{t} = \bigl[  \beta \Delta m_{t} + {\rm div} \bigl( m_{t} D_{p} H(Du_{t},x) 
\bigr) \bigr] dt 
- {\rm div} ( m_{t} \sqrt{2\beta} dW_{t} \bigr) 
 \ \  {\rm in } \ \   \R^d\times (0,T), \\[2mm]
 \ds u_T(x)=G(x,m_T) \ \ \   m_{0}=\bar m_0 \ \  {\rm in } \ \  \R^d.
\end{array}\right.
\ee

We introduce a notion of  solution of \eqref{e.MFGstoch}, which  is adapted to the common noise $W$,  and prove existence and uniqueness  when the couplings $F$ and $G$ are  nonlocal and
satisfy the well-known Lasry-Lions monotonicity condition introduced in \cite{LLJapan}.   Exact assumptions are stated later.  To the best of our knowledge, this is  the first general result for solutions of the MFG-system with common and no idiosynctratic noise. We also use the solution to derive approximate Nash equilibria  for $N$-player differential games with common but no idiosyncratic noise assuming  a structure condition on $H$.   
\smallskip

An important step in our analysis is the study of the well-posedness of strong, that is, a.e. in space-time  and semiconcave in space, solutions of the backward stochastic Hamilton-Jacobi (HJ for short) equation
\be\label{HJeqIntro}
\left\{\begin{array}{l}
 d_{t} u_{t} = \bigl[ - \beta \Delta u_{t} + H_t(Du_{t},x) -  \sqrt{2\beta} {\rm div}(v_{t}) \bigr] dt+ v_{t} \cdot \sqrt{2\beta}dW_{t}  \ \  {\rm in } \ \ \R^d\times (0,T), \\[2mm]
 u_T=  G \ \  {\rm in } \ \ \R^d,
\end{array}\right. 
\ee
with $H$ uniformly convex; again exact assumptions are stated later.
\smallskip

MFG  with common noise describe  optimal control problems with infinitely many small and interacting controllers whose dynamics are subjected to common noise. Such models  appear often  
in macroeconomics under the name  ``heterogenous agent models''; see, for instance, the  work of Krusell and Smith \cite{KrSm}. 
\smallskip

The mathematical description  of MFG with common noise, which  was introduced by Lions  \cite{LiCoursCollege} and then discussed, at an informal level,  in Bensoussan, Frehse  and Yam \cite{BeFrYa13} and Carmona and Delarue  \cite{CaDe14}, is  either probabilistic or analytic. 
\smallskip

%

The probabilistic formulation takes the  form of an optimal stochastic control problem involving a random distribution of the agents, which is the conditional law, given the common noise, of the optimal trajectory of the agents. 
In this setting, the existence of generalized solutions, that is,  solutions adapted to a larger filtration than the one generated by the common noise,  has been established  by Carmona, Delarue and Lacker \cite{CaDeLa} under very general assumptions; see also Ahuja \cite{ahuja2016wellposedness} and Lacker and Webster \cite{lacker2015translation}.  The former reference 
also establishes   the existence and uniqueness of a strong solution, which is adapted to the filtration of the common noise, under the Lasry-Lions monotonicity condition  and an  assumption on  the uniqueness of the optimal solution (with relaxed controls) of the stochastic control problem. Although the probabilistic  formulation allows for general existence and uniqueness results, it  requires the resolution of an optimal stochastic control problem, which is not  convenient for the numerical approximation or  the explicit analysis. Moreover, the condition on the uniqueness of the optimal solution, which  is quite demanding, has been shown to be satisfied  only under the strong assumption that either there is a nondegenerate idiosyncratic noise, or that the value function is  convex in space. The latter  is known to hold only for dynamics which are linear in space  and for cost functions which are strictly convex in space and control. 
\smallskip

The analytic approach to study  MFG problems with common noise involves the value function and the partial differential equation (PDE for short) it satisfies. There are two different but also very related  formulations  involving either stochastic MFG systems or the so-called master equation.

\smallskip

The former describes  the problem as a  coupled system of SPDEs known  as  the backward-forward stochastic MFG system. For problems without common noise, the system was introduced and studied by Lasry and Lions  in \cite{LL06cr1, LL06cr2}. In the presence of both idiosyncratic and common noises the  stochastic MFG system was first investigated by Cardaliaguet, Delarue, Lasry and Lions in \cite{CDLL}.  
\smallskip

The second analytic  approach, which was introduced by Lasry and Lions and  presented by Lions in \cite{LiCoursCollege},   is based on the  master equation, which is  a deterministic infinite dimensional PDE set  in the space of measures.
The  existence and uniqueness of solutions of the general infinite dimensional version of the master equation with idiosyncratic and common noise was shown in \cite{CDLL};  see also \cite{CaDeBook} for a generalization. Among other recent references about the master equation we point out  the works of Cardaliaguet, Cirant and Porretta  \cite{cardaliaguet2020splitting} who proposed a splitting method,  Lions \cite{LiCoursCollege} who introduced  the Hilbertian approach, in which   the master equation is embedded  in  the space of square integrable random variables, and, finally, Bayraktar, Cecchin, Cohen  and Delarue \cite{bayraktar2019finite} and Bertucci, Lasry  and Lions \cite{bertucci2019some, BeLaLi20} who  investigated  the existence and uniqueness for problems with common noise in finite state spaces.  
\smallskip

Here   we  study the stochastic MFG system \eqref{e.MFGstoch} which consists of 
a backward stochastic HJ-equation coupled with a forward stochastic Kolmogorov-Focker-Plank (KFP for short) PDE. In \eqref{e.MFGstoch},  
the Brownian motion  $W$ is the common noise and the unknown is the triplet  $(u,m,v)$, consisting of  
the value function $u$ of a  small agent, which  solves the backward HJ  SPDE  with Hamiltonian of the form $H(Du_t,x)-F(x,m_t)$, an auxiliary function $v$ which  ensures that  $u$ is  adapted to the filtration generated by the noise $W$, and  the density $m$  of the players which solves the forward stochastic KFP  SPDE. The two equations are coupled through  $F$ and $G$, which depend on $m$ in a nonlocal way. 
\smallskip

The main difference with previous works and, in particular, \cite{CDLL} is that, due to  the absence of idiosyncratic noise,  the solution of \eqref{e.MFGstoch} is not expected to  be ``smooth'' since  the HJ and KFP equations are only degenerate parabolic. 
\smallskip

To explain this problem, we first consider the HJ equation separately, that is, we look at \eqref{HJeqIntro}, which is a backward SPDE (BSPDE for short) associated with an optimal control problem with random coefficients. It follows from the work of 
Peng \cite{Pe92} that \eqref{HJeqIntro} has a unique solution provided that the noise satisfies a  nondegeneracy assumption, which, roughly speaking, means  that the $\beta$ in front of $ \Delta u_{t}$ is greater than the   $\beta$ in front of the terms involving $v$. More recently,  \eqref{HJeqIntro} was studied by Qiu \cite{QIU18} and Qiu and Wei \cite{QiuWei}, who introduced a notion of viscosity solution involving derivatives on the path space and proved its existence and  uniqueness. The equations studied in the last references  are  more general than \eqref{HJeqIntro}, in particular, the volatility is not constant, and require few  conditions on the Hamiltonian other than the standard growth and regularity.  
\smallskip

The study of  MFG with common noise necessitates the  use  of a completely different approach,  since the continuity equation for $m$ involves the derivative of $u$ in space and not a weak derivative in the path space. 
\smallskip

To motivate the new  approach we are putting forward here, we recall what happens for MFG problems without noise at all, that is, when $\beta=0$. In this deterministic case, where one can  take $v\equiv 0$, the natural concept of solution for \eqref{e.MFGstoch} requires  $u$  to be Lipschitz continuous and to satisfy the HJ equation in the viscosity sense, while  $m$ has to be bounded and to satisfy the KFP equation in the sense of distributions; see  \cite{LLJapan} and Cardaliaguet and Hadikhanloo   \cite{CaHa} for details. We note that, since  $m$ is absolutely continuous and bounded,  the term $mD_pH(Du_t,x)$ is  well-defined. However, it is also known that the boundedness condition on $m$ can  hold on large time intervals only if $H=H(p,x)$ is convex in $p$; see, for example,  Golse and James \cite{JG} where, to study  a  forward-forward system with nonconvex  $H$, it is  necessary to consider a much more degenerate notion of solution. When $H$ is convex in $p$, the solution 
of the HJ equation is naturally semiconcave.  Then the notion of viscosity  solution of  HJ is equivalent to the one  of the semiconcave a.e. solution  studied by Kruzhkov \cite{Kr60}; see also Douglis \cite{Do61}, Evans \cite{EvBook}, and Fleming \cite{Fl69}.
\smallskip

The first contribution of the paper is to show that the notion of semiconcave a.e. solution can be adapted to the BSPDE  \eqref{HJeqIntro} when reinterpreted in a suitable way. The starting point is the change of variable
 $$\tilde u_t(x)= u_t(x+\sqrt{2\beta} W_t),$$
which, using  the It\^o-Wentzell formula, leads, at least formally,  to
\be\label{BSDHJE_Intro}
d_t\tilde  u_t =  \tilde H_t (D\tilde u_{t},x)dt +d\tilde M_t \ \ \text{in} \ \ \R^d\times (0,T) \qquad \tilde u_T= \tilde G \ \ \text{in} \ \ \R^d,
\ee
with
\be\label{takis0} 
\tilde H_t(p,x)= H_t(p, x+\sqrt{2\beta} W_t) \ \ \text{ and} \ \  \tilde G(x)= G(x+\sqrt{2\beta} W_T).
 \ee

The problem then becomes to find a pair $(\tilde u_t,\tilde M_t)_{t\in [0,T]}$ adapted to the  filtration of $(W_t)_{t\in [0,T]}$,  where $\tilde M=\tilde M_t(x)$ is a globally bounded martingale, $\tilde u$ is continuous and semiconcave in the space variable, and \eqref{BSDHJE_Intro} is satisfied in an integrated form. 
\smallskip

Theorem \ref{thm.mainexists} and Proposition \ref{prop.comparison} establish respectively  that \eqref{BSDHJE_Intro} has a  solution and a comparison principle is satisfied. In  
Proposition \ref{prop.repsol} we also provide a stochastic  control representation of the solution.
The optimality conditions 
and the existence and the uniqueness of optimal trajectories are 
respectively the topics of Theorem \ref{thm.MaxPple} 
and Proposition \ref{prop.ExistOptTraj}. However, this last  point, which relies on the analysis of a continuity equation  associated with the drift $-D_p\tilde H(D\tilde u_t(x),x)$ (see Proposition~\ref{prop.uniConti} and Proposition~\ref{prop.UnikCont}) requires a much stronger structure condition on the Hamiltonian, which  we discuss  later. 
\smallskip 

Next we apply this approach to  \eqref{e.MFGstoch}. After the formal change of variables 
\be\label{takis101}
\tilde u_t(x)= u_t(x+\sqrt{2\beta}W_t, x) \ \ \text{and}  \ \ \tilde m_t=(id-\sqrt{2\beta} W_t)\sharp m_t,
\ee
we obtain  the new system 
\be\label{stoMFG_Intro}
\left\{\begin{array}{l}
\ds d_t\tilde  u_t = \left[ \tilde H_t (D\tilde u_{t}(x),x)-\tilde F_t(x,\tilde m_t)\right]dt +d\tilde M_t \ \    {\rm in} \  \ \R^d\times (0,T), \\[2mm] 
\ds \partial_t\tilde m_t =  {\rm div}(\tilde m_tD_p\tilde H(D\tilde u_t(x),x))dt \ \  {\rm in} \ \ \R^d\times (0,T), \\[2mm] 
 \ds \tilde m_0=\bar m_0 \qquad \tilde u_T = \tilde G(\cdot,\tilde m_T) \ \  {\rm in}\ \  \R^d \ \ \ 
\end{array}\right.
\ee
with $\tilde H$ as in \eqref{takis0},
 and 
$$\tilde F_t(x,m)= F(x+\sqrt{2\beta}W_t,(id+\sqrt{2\beta}W_t)\sharp m) \ \text{ and}   \ \tilde G(x)= G(x-\sqrt{2\beta} W_T, (id+\sqrt{2\beta}W_T)\sharp m_T).$$

 This transformation  was used in \cite{CDLL} to prove the existence of a strong solution of the stochastic MFG system with  common and  idiosyncratic noises, in which case \eqref{stoMFG_Intro} is non degenerate and has a ``smooth'' solution. 
\smallskip

Coming back to the degenerate system \eqref{stoMFG_Intro}, the problem is to find a solution

 $(\tilde u_t,\tilde M_t, \tilde m_t)_{t\in [0,T]}$ which is adapted to the  filtration  $(\mathcal F_t)_{t\in [0,T]}$ generated by $W$ and  $\tilde M=\tilde M_t(x)$ is a $({\mathcal F}_t)_{t\in [0,T]}$ martingale.  By a solution, we mean that  $\tilde u$ is space-time continuous and semiconcave in space,  while the martingale $\tilde M$ is  bounded, the HJ equation, once integrated in time, is satisfied $\P-$a.s. and a.e., and  the random measure $\tilde m$  has globally  bounded density satisfying   the continuity equation in the sense of distribution $\P-$a.s..  
\smallskip

Our main result, Theorem \ref{thm.main},  is that, under suitable structure and  regularity assumptions on the data and assuming that $F$ and $G$ are strongly monotone in the Lasry-Lions sense, 
\eqref{stoMFG_Intro} has a unique solution $(\tilde u, \tilde M, \tilde m)$. 
This is the first existence and uniqueness result of a  strong solution for MFG problems with a common and without idiosyncratic noise with value function that  is neither smooth nor convex in space in contrast with \cite{CDLL, CaDeLa}.
\smallskip

The final result concerning MFG is Proposition~\ref{prop.cv}. It  asserts that it is possible to use the solution of  the stochastic MFG system \eqref{e.MFGstoch} to derive approximate Nash equilibria in $N-$player differential games with a common noise. Such a  statement is standard in the MFG literature. The first results in this  direction go back to  Huang, Caines and Malham\'e \cite{huang2003individual}, \cite{huang2006large} 
for linear and nonlinear dynamics respectively. In both these references, the dynamics and payoff depend on  the empirical measure  through an average. Hence, the Central Limit Theorem  implies that the error term is of order $N^{-1/2}$. The result for a genuinely non linear version of MFG problems without common noise was obtained by Carmona and Delarue \cite{carmona2013probabilistic}; see also \cite{CaDe14}, Section~6 in Vol. II. Since then, there have been many variations and extensions, and we  we refer to \cite{CaDe14} and the references therein. As far as we know, Proposition \ref{prop.cv} is the first result for MFG problems with common and without idiosyncratic noises.  The main reason for proving it  is that it justifies the (somewhat formal) change of variables made to pass from the original MFG system \eqref{e.MFGstoch} to the transformed one \eqref{stoMFG_Intro}.  Indeed, we use  \eqref{stoMFG_Intro} to solve a problem which should actually involve the solution of \eqref{e.MFGstoch}. 
\smallskip

In contrast with  the analysis of the MFG system, to find the approximate Nash equilibria we need to consider  a  special  class the Hamiltonian $\tilde H$. Indeed,  we assume that, for some smooth and strictly positive coefficient $\tilde a$ and a smooth and  bounded vector field $\tilde b$,  $\tilde H$ is of the form 
\be\label{takis1}
\tilde H_t(p,x)= \frac12 \tilde a_t(x) |p|^2 + \tilde b_t(x)\cdot p.
\ee
We  suspect that \eqref{takis1}  may not be necessary. The reason we have to require it is to obtain the uniqueness, for fixed $\omega$ and given the bounded variation vector field $-D_pH(D\tilde u_t(x),x)$,  of the solution of the continuity equation in \eqref{stoMFG_Intro}. 
\smallskip

The uniqueness of solutions of linear transport and continuity equations under weak assumptions on the vector field is a very intriguing problem. Its study goes back to DiPerna and Lions \cite{ DiPLio} and Ambrosio \cite{Am04}.  These results cannot be applied to the case at hand, since they require regularity which is not satisfied by $-D_pH(D\tilde u_t(x),x)$.  Instead, here we rely on a result of Bouchut, James and Mancini \cite{BoJaMa} which requires a half-Lipschitz condition on the vector field. To use it, however, here we need to assume \eqref{takis1}. 
\smallskip

\subsection*{Organization of the paper} The paper is organized in two parts. In the first, we study the backward  HJ SPDE  \eqref{HJeqIntro}.  We state the assumptions in subsection \ref{subsec.ass}, 
show the existence of a solution in subsection \ref{subseq.exHJ}, prove its  uniqueness by a comparison principle in subsection \ref{subseq.Comp}, propose an optimal control representation  and discuss a maximum principle in subsection \ref{subseq.OCrep}. In order to prove the existence of optimal solutions in subsection \ref{subsec.exOS},  we first  discuss in subsection \ref{subseq.coneq}
 the continuity equation associated for the optimal drift. 
 The second part is devoted to the stochastic MFG system \eqref{e.MFGstoch}.   We state the assumptions in subsection~\ref{subsec.exTakis}, and the main existence and uniqueness result in subsection \ref{subsec.exMFG}.  In subsection \ref{subsec.MFGwn} we recall the case without noise, for which we provide sharp estimates.  We then construct  approximate solutions of the stochastic MFG system in subsection \ref{subsec.appMFG} and, finally, pass to the limit to prove the main result in subsection \ref{subsec.lim}. In subsection \ref{subsec.game} we show the existence of the approximate Nash equilibria for the $N$-player game.
Finally,  in the appendix we revisit the result of \cite{BoJaMa} on the uniqueness of the solution to some continuity equations. 
\subsection*{Notation} Throughout  the paper $\mathcal O$ is an open subset of $\R^d$, and $C^2(\mathcal O)$ is the space of $C^2$-maps on $\mathcal O$ with bounded derivatives  endowed with the sup-norm 
$$
\|u\|_{C^2(\mathcal O)}= \|u\|_{L^\infty(\mathcal O)}+ \|Du\|_{L^\infty(\mathcal O)}+\|D^2u\|_{L^\infty(\mathcal O)}; 
$$
depending on the context,  we often omit the subscript $L^\infty(\mathcal O)$ and simply write $\|\cdot\|_\infty.$
We work on a  complete filtered probability space $(\Omega, \mathcal F, (\mathcal F_t)_{t\geq 0}, \P)$ 
carrying an d-dimensional Wiener process $W = (W_t)_{t \in [0,T]}$ such that $(\mathcal F_t)_{t\geq0}$ is the natural filtration generated by $W$ augmented by all the $\P-$null sets in $\mathcal F$.
We denote by $\mathcal P$ the $\sigma-$algebra of the predictable sets on $\Omega\times [0,T]$  associated with $(\mathcal F_t)_{t\geq 0}$.  Given   a complete metric space $E$ and $p\geq 1$, ${\mathcal S}^p(E)$ is the space of continuous, $E$-valued, $\mathcal P-$measurable processes $X=(X_t)_{t\geq 0}$ such that, for some $x_0\in E$  and, therefore, any point in $E$, 
$$
\E[\sup_{t\in [0,T]} d_E(X_t,x_0)^pdt]<+\infty.
$$
We set $\mathcal S^r(C^2_{loc}(\R^d))= \bigcap_{n\geq 1} \mathcal S^r(C^2(B_n))$, where $B_n$ is the open ball centered at $0$ and of radius $n$, and define similarly $\mathcal S^r(L^1_{loc}(\R^d))$ and $\mathcal S^r(W^{1,1}_{loc}(\R^d))$. We write $L^\infty((\Omega \times \mathcal F_T); C^2(\R^d))$ for the the space of bounded $C^2(\R^d)$-valued and $\mathcal F_T$ maps on $\Omega$. For $k\geq 1$, $\mathcal P_k(\R^d)$ denotes the set of Borel probability measures on $\R^d$ with finite $k-$th order moment $M_k(m)= \int_{\R^d} |x|^k m(dx)$ , endowed with the Wasserstein distance ${\bf d}_k$. Finally, $\nu_y$ denotes the external normal vector to a ball $B_r$ at $y \in \partial B_r$.

\subsection*{Some assumptions and terminology} To ease the presentation and avoid repetitions in the rest of the paper, we summarize here some of the terminology we use and the assumptions make. 
\smallskip

A map $\mathcal G:\R^d\times \mathcal P_1(\R^d)\to \R$ is called strongly monotone, if there exists $\alpha>0$ such that,
for all $m_1,m_2 \in \mathcal P_1(\R^d)$
\begin{equation}\label{SM1}
\displaystyle \int_{\R^d} (\mathcal G (x,m_1)-\mathcal G(x,m_2)) (m_1-m_2)(dx) 
 \geq \alpha \int_{\R^d} (\mathcal G(x,m_1)-\mathcal G(x,m_2))^2dx. 
\end{equation}
A map $\mathcal G:\R^d\times \mathcal P_1(\R^d)\to \R$ is called strictly monotone, if 
\begin{equation}\label{SM2}
\displaystyle \int_{\R^d} (\mathcal G (x,m_1)-\mathcal G(x,m_2)) (m_1-m_2)(dx)\leq 0 \ \  \text{implies} \ \ m_1=m_2.
\end{equation}

The typical assumption required for  Hamiltonians $\mathcal H:\R^d\times \R^d \to \R$ we consider in this paper is that 
\be\label{HH}
\begin{cases}
(i) \ \mathcal H={\mathcal H}(p,x) \ \text{is convex in $p$}, \\[2mm]
(ii) \; \text{for any $R>0$,  there exists $C_R>0$ such that,}\\[2mm]
\text{ for all $x,p\in \R^d$ with  $|p|\leq R$, }\\[2mm]
\qquad  |\mathcal H (p,x)|+ |D_p \mathcal H(p,x)|+ |D^2_{px}\mathcal H(p,x)|+ |D^2_{pp}\mathcal H(p,x)|\leq C_R, \\[2mm]
(iii)\; \text{ there exists $\lambda >0$ and $C_0>$ such that,}\\[2mm] 
\text{for any $p,q,x,z\in \R^d$ with $|z|=1$ and in the sense of distributions, }\\[2mm]
\lambda(D_p\mathcal H(p,x)\cdot p- \mathcal H(p,x))+ D^2_{pp}\mathcal H(p,x)q\cdot q \\[1.5mm]
  \qquad \qquad \qquad \qquad \qquad \qquad  + 2D^2_{px}\mathcal H(p,x)z\cdot q + D^2_{xx}\mathcal Hp,x) \ z\cdot z \geq - C_0.
\end{cases}
\ee
The typical regularity assumption we will need for maps $\mathcal G:\R^d\times \mathcal P_1(\R^d)\to \R$ is 
\be\label{FG}
\begin{cases}
 \mathcal G\in C(\R^d\times \mathcal P_1(\R^d);\R) \ \text{and 
there exists $C_0>0$ such that}\\[2mm]
\underset{m\in \mathcal P_1(\R^d), \; t\in [0,T]}\sup\left[
\|\mathcal G(\cdot, m)\|_\infty +\|D\mathcal G(\cdot ,m)\|_\infty + \|D^2\mathcal G(\cdot ,m)\|_\infty\right] \leq C_0,
\end{cases}
\ee

\subsection*{Acknowledgments} Cardaliaguet was partially supported by the  AFOSR grant FA9550-18-1-0494. Souganidis was partially supported by the NSF grants DMS-1600129 and DMS-1900599, the ONR grant N000141712095 and the AFOSR grant FA9550-18-1-0494.
%
%
%
%
%
%
%

\section{The stochastic backward Hamilton-Jacobi equation}

Following the discussion in the introduction,  here we study  the HJ SPDE \eqref{BSDHJE_Intro}
with $\tilde H_t$ and $\tilde G$ given by \eqref{takis0} and $\tilde M$ an unknown martingale. 

\subsection{Assumptions and the notion of solution}\label{subsec.ass}
We introduce the main assumptions about the continuity of $\tilde H$ and $\tilde G$, the convexity and coercivity in the gradient and the higher regularity in the space and gradient variables of $\tilde H$ (\eqref{H1} and \eqref{H2}), and the continuity in time of $\tilde H$ \eqref{H3}. To simplify the notation, when possible, we omit the explicit dependence on $\omega$. 
\smallskip

We assume that
\be\label{H1}
\tag{H1}
\tilde G \in L^\infty(\Omega, \mathcal F_T, C^2(\R^d)) \ \ \text{ and}  \ \ \tilde H\in \mathcal S^r(C^2_{loc}(\R^d\times \R^d)) \ \ \text{ for all \ $r\geq 1$,} 
\ee
and
\be\label{H2}
\tag{H2}
\tilde H_t \ \text{satisfies \eqref{HH} uniformly in $t\in [0,T]$ and in $\omega\in \Omega$.}  
\ee

To quantify the continuity of $\tilde H$ in time, we define, for $R>0$, 
\be\label{takis8}
\omega^N_R = \sup_{|p|\leq R, \ |s-t|\leq 1/N,\ y\in \R^d} |\tilde H_s(p,y)-\tilde H_t(p,y)|.
\ee
We assume that, for all  $R>0$, 
\be\label{H3}
\tag{H3}
\underset{N\to \infty} \lim \E[\omega^N_R]=0.
\ee

The notion of solution of \eqref{BSDHJE_Intro} is introduced next.

\begin{defn}\label{def.solHJ} The couple  $(\tilde u,\tilde M):\R^d\times [0,T]\times \Omega \to \R^2$ is a solution of  \eqref{BSDHJE_Intro} if the following conditions hold: 
\smallskip

(i)~$\tilde u\in \mathcal S^r(W^{1,1}_{loc}(\Rd))$ and  $\tilde M\in \mathcal S^r(L^1_{loc}(\R^d))$ for all $r\geq 1$, 
\smallskip

(ii)~ there exists $C>0$ such that, $\P-$a.s., for a.e. $t\in [0,T]$ and all $z\in \R^d$ such that $|z|\leq 1$ and in the sense of distributions, 
$$
 \|\tilde u_t\|_\infty + \|D\tilde u_t\|_\infty +  \|\tilde M_t\|_\infty  + D^2\tilde u_t \ z\cdot z \leq C,
 $$
(iii)~  for a.e. $x\in \R^d$, the process $(\tilde M_t(x))_{t\in [0,T]}$ is a continuous  martingale,
\smallskip

and
\smallskip

(iv)~  for a.e. $(x,t)\in \R^d\times [0,T]$ and $\P-$a.s., 
\be\label{def.eq}
\tilde  u_t(x) = \tilde G(x) -\int_t^T \tilde H_s(D\tilde u_s(x),x)ds -\tilde M_T(x)+ \tilde M_t(x).
\ee
\end{defn}

\begin{rmk}{\rm As it is often the case in the literature, the martingale $\tilde M_t(x)$ can be written as a stochastic integral of an adapted process $\tilde Z_t(x)$. 
}\end{rmk}

\subsection{Existence of a solution}\label{subseq.exHJ}
We prove  that  \eqref{BSDHJE_Intro} has a solution as in Definition~\ref{def.solHJ}.

\begin{thm}\label{thm.mainexists} Assume \eqref{H1}, \eqref{H2} and \eqref{H3}. Then there exists a solution of  \eqref{BSDHJE_Intro}. 
\end{thm}


\begin{proof} The solution is obtained as the limit of solutions of a sequence of approximate problems which we introduce next. 
\smallskip

For each $N \in \N$,  we consider the partition of ${(t^N_n)}_{n\in \{0,\dots,N\}}$ of $(0,T)$ with $t^N_n=Tn/N$, set 
$$
\tilde H^N_t(p,x)= \tilde H_{t_n}(p,x)\ \ {\rm on} \ \ [t^N_n,t^N_{n+1}), 
$$
denote by $(\mathcal F^N_t)_{t\in [0,T]}$ the piecewise constant filtration given by 
$$\mathcal F^N_t= \mathcal F_{t^N_n} \ \  \text{if} \ \  t\in [t^N_n, t^N_{n+1}),$$ 
define using backward induction the  c\`adl\`ag in time processes $\tilde u^N=\tilde u^N_t(x)$ and  $\Delta M^N= \Delta M^N_{t_n}(x)$ by $\tilde u^N_T= \tilde G \ \ {\rm in} \ \  \R^d,$ and, for $n=0,\ldots, N-1$, 
\[
-\partial_t \tilde u^N_t+ \tilde H^N_t(D\tilde u^N_t,x)= 0 \ \  {\rm in } \ \  \R^d\times (t^N_n, t^N_{n+1}),\]
and, for $x\in \R^d$,
\[u^N_{t^{N,-}_{n+1}}(x)= \E\left[\tilde  u^N_{t^{N,+}_{n+1}}(x) \ |\ {\mathcal F}_{t_n}\right] \  \ \text{and} \ \  
\Delta \tilde M^N_{t^{N,-}_{n+1}}(x) = u^N_{t^{N,+}_{n+1}}(x)- \E\left[ u^N_{t^{N,+}_{n+1}}(x) \ |\ {\mathcal F}_{t^N_n}\right],\]
%
and, finally,  set 
$$
\tilde M^N_t(x) = \sum_{t_n\leq t} \Delta \tilde M^N_{t_n}(x).
$$
%
%
Using our assumptions, we have  the following lemma. Its  proof is presented  after the end of the ongoing one. 
\begin{lem}\label{lem.reguuNMN} Assume \eqref{H1}, \eqref{H2}, and \eqref{H3}. There exists a $C>0$ such that,  $\P-$a.s., a.e. $ t\in [0,T]$ and for all  $z\in \R^d$ with $|z|\leq 1$, and in the sense of distributions,
$$\|\tilde u^N\|_\infty + \|D\tilde u^N\|_\infty  +\|\tilde M^N\|_\infty + D^2\tilde u^N \ z\cdot z \leq C.$$ 
%
Moreover, for any $x\in \R^d$, the process $(\tilde u^N_t(x))_{t\in [0,T]}$ is adapted to the filtration $(\mathcal F^N_t)_{t\in [0,T]}$  and 
$(\tilde M^N_{t_n}(x))_{t\in [0,T]}$ is a  martingale with respect to the discrete filtration $(\mathcal F_{t_n})_{n=0,\ldots, N}$.
\end{lem}

It is immediate from the definition of the filtration $(\mathcal F^N_t)_{t\in [0,T]}$ that $(\tilde u^N_t(x))_{t\in [0,T]}$ is also adapted to $(\mathcal F_t)_{t\in [0,T]}$. 

Continuing with the ongoing proof we  note that $(\tilde u_t^N,\tilde M_t^N)_{t\in [0,T]}$ solves the backward equation
\be\label{lkjnzrdgf}
d\tilde u^N_t = \tilde H^N_t(D\tilde u^N_t,x) dt +d\tilde M^N_t \ \ \text{in} \ \ \R^d\times (0,T)  \quad \tilde u^N_T=\tilde G \ \ \text{in} \ \ \R^d, 
\ee
in the sense that, $\P-$a.s. and for a.e. $(x,t)\in \R^d\times [0,T]$,  
$$
\tilde u^N_t(x) =\tilde G(x) - \int_t^T \tilde H^N_s(D\tilde u^N_s(x),x) ds - \tilde M^N_T(x) + \tilde M^N_t (x). 
$$
We show next that   $(\tilde u_t^N,\tilde M_t^N)_{t\in [0,T]}$ is  Cauchy sequence in a suitable space, and for this we follow Douglis' \cite{Do61} uniqueness proof (see also \cite{EvBook}). 
\smallskip

Fix  $0<N<K$, let $\phi:\R\to [0,\infty) $ be a smooth, Lipschitz continuous, convex and nonincreasing map, and  set  
$$w_t(x)= \phi(\tilde u^N_t(x)-\tilde u^K_t(x)).$$ 
Using induction and the  convexity of $\phi$ to cancel the jump terms, which are martingales, we find that, for any $t\in [0,T)$ and $h\in(0,T-h)$, 
\be\label{takis2}
\begin{cases}
\E\Bigl[ w_{t+h}(x)- w_t(x) \Bigr]  \geq  \\[2.5mm]
\E\Bigl[  \int_t^{t+h} \phi'(\tilde u^N_s(x)-\tilde u^K_s(x)) (\tilde H^N_s(D\tilde u^N_s,x)-\tilde H^K_s(D\tilde u^K_s,x)) ds \Bigr]\\[2.5mm]
 \hskip1.25in = \E\Bigl[   \int_t^{t+h} ({\bf b}_s(x)\cdot D w_s(x) + \zeta_s(x))ds\Bigr] 
\end{cases}
\ee
where 
\[ \zeta_s(x)=  \phi'(\tilde u^N_s(x)-\tilde u^K_s(x))(\tilde H^N_s(D\tilde u^K_s(x),x)-\tilde H^K_s(D\tilde u^K_s(x),x)),\]
and
\[{\bf b}_s(x)= \int_0^1 D_p\tilde H^N_s((1-\lambda)D\tilde u^N_s(x)+\lambda D\tilde u^K_s(x),x)d\lambda.\] 
Let ${\bf b}^\ep$ be a regularization of $b$ to be specified below. Then \eqref{takis2} can be rearranged to read 
\be\label{lieznrdgf}
\E\Bigl[ w_{t+h}(x)- w_t(x) \Bigr]  \geq \E\Bigl[ \int_t^{t+h} {\rm div}({\bf b}^\ep_s w_s) -{\rm div}({\bf b}^\ep_s) w_s +({\bf b}_s-{\bf b}^\ep_s)\cdot Dw_s+ \zeta_sds\Bigr].
\ee

For $\alpha,\beta>0$ to be chosen, we consider the quantity 
\[
e_t = \E[ \int_{B_{\alpha+\beta t}} w_tdx],
\]
\vskip-0.1in
and  claim that 
\be\label{hjlbensrdc}
\begin{cases}
e_T-e_t \geq   \E\Bigl[ \beta \int_t^T \int_{\partial B_{\alpha+\beta s}} w_s(x) \ dxds\Bigr]\\[2.5mm]
\hskip.25in +  \E\Bigl[ \int_t^{T}\int_{B_{\alpha+\beta s}} {\rm div}({\bf b}^\ep_s w_s) -{\rm div}({\bf b}^\ep_s) w_s +({\bf b}_s-{\bf b}^\ep_s)\cdot Dw_s+ \zeta_sdxds\Bigr].
\end{cases}
\ee
Indeed, let $k$ be a large integer and set $\theta^k_r= t+ \frac{r}{k}(T-t)$ for $r\in \{0, \dots, k\}$. Integrating \eqref{lieznrdgf} 
over $B_{\alpha+\beta \theta^k_r}$ with $t=\theta^k_r$ and $h=1/k$ and summing over $r$ we obtain
\begin{align*} 
& 
e_t-e_t - \E\Bigl[\sum_{r=0}^{k-2} \int_{B_{\alpha+\beta \theta^k_{r+1}}\backslash B_{\alpha+\beta \theta^k_{r}}} w_{\theta^k_{r+1}}(x) dx\Bigr]
= \E\Bigl[\sum_{r=0}^{k-1}  \int_{B_{\alpha+\beta \theta^k_r}} (w_{\theta^k_{r+1}}(x)- w_{\theta^k_r}(x))dx \Bigr]\\
& \geq 
 \E\Bigl[\sum_{r=0}^{k-1}   \int_{\theta^k_r}^{\theta^k_{r+1}} \int_{B_{\alpha+\beta \theta^k_r}}
( {\rm div}({\bf b}^\ep_s w_s) -{\rm div}({\bf b}^\ep_s) w_s +({\bf b}_s-{\bf b}^\ep_s)\cdot Dw_s+ \zeta_s)dyds\Bigr],
\end{align*} 
and,  after letting $k\to +\infty$, \eqref{hjlbensrdc}. 
\smallskip

Next, rearranging \eqref{hjlbensrdc} we find, for all $t\in (0,T)$. 
\[
\begin{cases} e_t  \leq  \E\Bigl[  - (\int_t^T\int_{B_{\alpha+\beta s}}  ({\rm div}({\bf b}^\ep_s w_s) -{\rm div}({\bf b}^\ep_s) w_s +({\bf b}_s-{\bf b}^\ep_s)\cdot Dw_s + \zeta_s) dyds \\[3mm]
 \qquad - \beta \int_t^T \int_{\partial B_{\alpha+\beta s}} w_sdx ds   \Bigr]+e_T,
\end{cases} 
 \]
and, after integrating by parts, 
\be\label{takis10}
\begin{cases}
e_t   \leq  
 \E\Bigl[  -\int_t^T \int_{B_{\alpha+\beta s}}( 
 -{\rm div}({\bf b}^\ep_s) w_s +({\bf b}_s-{\bf b}^\ep_s)\cdot Dw_s + \zeta_s) dyds \\[2mm]
 \qquad - \int_t^T \int_{\partial B_{\alpha+\beta s}}  ({\bf b}^\ep_s\cdot \nu_y + \beta) w_s dyds \Bigr] +e_T.
\end{cases}
\ee

We return now to the choice of ${\bf b}^\ep$. For this, let $\tilde u^{N,\ep}=\tilde u^{N}\ast \xi^\ep$ and $\tilde u^{K,\ep}=\tilde u^{K}\ast \xi^\ep$ be  space-time regularizations of $\tilde u^N$ and $\tilde u^K$  with a smooth compactly supported kernel $\xi^\ep$. 
\smallskip

Then, for all  $(x,t) \in \R^d\times (0,T)$,  $z\in \R^d$ such that $|z|\leq 1$ and in the sense of distributions, 
\be\label{takis3}
\begin{cases}
|\tilde u^{N,\ep}(x,t)|\leq \|\tilde u^{N}\|_\infty, \ \  |\tilde u^{K,\ep}(x,t)|\leq \|\tilde u^{K}\|_\infty, \\[2mm]
|D\tilde u^{N,\ep}(x,t)|\leq  \|D \tilde u^{N}\|_\infty,\ \  |D\tilde u^{K,\ep}(x,t)|\leq  \|D\tilde u^{K}\|_\infty,\\[2mm]
D^2\tilde u^{N,\ep}(x,t) \ z \cdot z \leq C, \ \  D^2\tilde u^{K,\ep}(x,t) \  z\cdot z \leq C,
\end{cases}
\ee
and,  as $\ep\to 0$ and for a.e. $(x,t)$,
\be\label{takis4}
D\tilde u^{N,\ep}(x,t) \to Du^{N}(x,t) \ \ \text{and} \ \ D\tilde u^{K,\ep}(x,t) \to Du^{K}(x,t).
\ee
Let 
$$
{\bf b}^\ep_s(x)= \int_0^1 D_p\tilde H^N_s((1-\lambda)D\tilde u^{N,\ep}_s(x)+\lambda D\tilde u^{K,\ep}_s(x),x)d\lambda. 
$$
It is immediate from the properties of $\tilde H$,  \eqref{takis3} and Lemma~\ref{lem.reguuNMN} that there exists  $C_1>0$ such that, for all $(x,t)$ and $\P-$a.s. in $\omega$, 
\be\label{takis6}
|{\bf b}^\ep_s(x)| \leq \sup_{|p|\leq C, \ y\in \R^d, \ \omega\in \Omega} |D_p\tilde H^N_s(p,y,\omega)| \leq C_1,
\ee
where $C$ is the upper bound on $\|D\tilde u^N\|_\infty$ and $\|D\tilde u^K\|_\infty$ in  Lemma~\ref{lem.reguuNMN}. 
\smallskip

Furthermore, as $\ep\to 0$ and $\P-$a.s., ${\bf b}^\ep \to {\bf b}$ for a.e. $(x,t)$ and in any $L^p_{loc}$.
\smallskip

Finally, since
\begin{align*}
{\rm div}({\bf b}^{\ep}_t(x))& =  \int_0^1 {\rm Tr}(D^2_{pp}\tilde H^N_{t}((1-\lambda)D\tilde u^{N,\ep}+\lambda D\tilde u^{K,\ep},x)((1-\lambda)D^2\tilde u^{N,\ep}+\lambda D^2\tilde u^{K,\ep}))d\lambda \\
& + \int_0^1 {\rm Tr}(D^2_{px}\tilde H^N_{t}((1-\lambda)D\tilde u^{N,\ep}+\lambda D\tilde u^{K,\ep},x))d\lambda, 
\end{align*}
it follows from the convexity of $\tilde H$, \eqref{H2} with $R=C$ from Lemma~\ref{lem.reguuNMN}, and  \eqref{takis3} that there exists $\tilde C>0$ such that 
$$
{\rm div}({\bf b}^{\ep}_t(x))\leq \tilde C. 
$$

We choose $\beta = C_1$ in \eqref{takis10}. Recalling that $w\geq 0$,  we find, for some other $C>0$,  
\begin{align*}
e_t &\leq  \E\Bigl[ \int_t^T \int_{B_{\alpha+\beta s}}(  \tilde Cw_s -({\bf b}_s-{\bf b}^\ep_s)\cdot Dw_t -  \zeta_s) dyds\Bigr] +e_T \\
& \leq C\int_t^T e_sds +\E\Bigl[ \int_t^T \int_{B_{\alpha+\beta s}}(  -({\bf b}_s-{\bf b}^\ep_s)\cdot Dw_t -  \zeta_s) dyds\Bigr] +e_T.
\end{align*}
Using Gronwall's inequality and letting $\ep\to 0$, we get
\be\label{takis11}
e_t\leq C (e_T+ \E\Bigl[\int_t^T \int_{B_{\alpha+\beta s}} \zeta_s(x)dxds\Bigr]).
\ee
Next we note that, since, in view of Lemma~\ref{lem.reguuNMN},  $\|Du^N\|\leq C$, if $\omega^N_C$ is as in \eqref{takis8}, we have 
$$
\|\zeta_s\|_\infty\leq \|\phi'\|_\infty \omega^N_C.
$$

Assume next that $\phi$ is positive on $(-\infty,0)$, vanishes on $(0,+\infty)$ and  $\|\phi'\|_\infty\leq 2$. Since $\tilde u^N_T= \tilde u^K_T$, we have   $e_T=0$.  
\smallskip

Therefore, it follows from \eqref{takis11} and the above that, for all $t\in [0,T]$,  
$$
e_t\leq C_{\alpha,\beta} \E[\omega^N_C]. 
$$ 
Note that a standard  approximation argument implies  the same inequality for $\phi(s)= (-s)_+$.
\vskip.075in

Hence, for all $t\in [0,T]$,  we have  
$$
\E\left[ \int_{B_{\alpha+\beta t}} (- (\tilde u^N_t(x)- \tilde u^K_t(x))_+dx\right] \leq  C_{\alpha,\beta} \E[\omega^N_C],
$$
and, after exchanging the roles of $u^N$ and $u^K$, for all $t\in [0,T]$,  
$$
\E\left[ \int_{B_{\alpha+\beta t}} |\tilde u^N_t(x)- \tilde u^K_t(x)| dx\right] \leq  C_{\alpha,\beta} \E[\omega^N_C].
$$
Since $\alpha$ is arbitrary and the $\tilde u^N$'s are uniformly bounded and uniformly Lipschitz continuous in space, the   inequality $$\|u\|_{L^\infty(B_R)}\leq C_R \|Du\|_{L^\infty(B_R)}^{d/(d+1)}\|u\|^{1/(d+1)}_{L^1(B_R)}$$ yields that,
for all $R>0$, 
$$
\sup_{t\in [0,T]} \E\left[ \|\tilde u^N_t- \tilde u^K_t\|_{L^\infty(B_R)} ^{d+1}\right] \leq  C_R \E[\omega^N_C].
$$
In view of \eqref{H3}, it follows that $(\tilde  u^{N})_{N\in \N}$ is a Cauchy sequence for the family of seminorms 
$$
(\sup_{t\in [0,T]} \E l[\|\tilde u_t\|_{L^\infty(B_R)}^{d+1}])_{R>0}. 
$$
Thus, there is a subsequence, which we denote in the same way as the full sequence, along which the $\tilde  u^{N}$'s converge, in the seminorms above and for every $R>0$, to  a limit $\tilde u$, which, in view of the uniform estimates in Lemma~\ref{lem.reguuNMN}, 
is  Lipschitz continuous and semiconcave in $x$. Moreover, the process $(\tilde u_t(x))_{t\in [0,T]}$ is adapted to the filtration $({\mathcal F}_t)_{t\in [0,T]}$. Finally, up to a further subsequence, we can assume that $\tilde u^N$ converge  to $\tilde u$  locally uniformly in $x$,  $\P-$a.s. and a.e. $t\in [0,T]$. In view of the  uniform semiconcavity, the last observation  implies that, as $N\to \infty$ and $\P-$a.s. and for a.e. $(x,t)\in \R^d\times [0,T]$,  $D\tilde u_t^N(x) \to D\tilde u_t(x)$.
\smallskip

Let  $\omega\in \Omega$ and $t\in [0,T]$ be such that, as $N\to \infty$,  $\tilde u^N_t(\cdot,\omega)$ converges locally uniformly to $\tilde u_t(\cdot, \omega)$.  Integrating  \eqref{lkjnzrdgf}  over $[t,T]$, we then find, using Fubini's theorem,  that, for a.e. $x\in \R^d$, 
$$
\tilde u^N_t(x)=\tilde G(x) - \int_t^T \tilde H^N_s(D\tilde u_s^N(x),x) ds +\tilde M^N_t(x).  
$$
Since  $D\tilde u^N$ converges a.e. to $D\tilde u$ on $\R^d\times (0,T)$ and is bounded, we can pass in  the $N\to \infty$ limit in the equality above to get, for a.e. $x\in \R^d$,
\be\label{liauzqesnrdgc}
\lim_{N\to +\infty} \tilde M^N_t(x)
= 
 \tilde u_t(x)- \tilde G(x)  + \int_t^T \tilde H^N_s(D\tilde u_s^N(x),x) ds.
\ee
Hence, the $\tilde M^N$'s converge $\P-$a.s. and for a.e. $x\in \R^d$ to some bounded process denoted by $\tilde M$ and we have,  for a.e. $x\in \R^d$, 
$$
\tilde M_t(x) = \tilde u_t(x)- \tilde G(x)  + \int_t^T \tilde H^N_s(D\tilde u_s^N(x),x) ds.
$$
Since  $(\tilde M^N_{t_n}(x))_{t\in [0,T]}$ is a martingale in the filtration $(\mathcal F_{t_n})_{n=0,\ldots, N}$, it follows that $(\tilde M_t(x))_{t\in [0,T]}$ is, for a.e. $x\in \R^d$, a martingale in the filtration $(\mathcal F_t)_{t\in [0,T]}$. In particular, for a.e. $x\in \R^d$, $t\to \tilde M_t(x)$ is continuous, which shows that $t\to \tilde u_t(x)$ is continuous as well. 
Therefore $(\tilde u, \tilde M)$ is a solution to \eqref{BSDHJE_Intro}. 
\end{proof}

We conclude this subsection with the proof of Lemma \ref{lem.reguuNMN}. 

\begin{proof}[Proof of Lemma \ref{lem.reguuNMN}]
Since the estimates are standard, here  we explain only the formal ideas. As usual the  computations can be justified by vanishing viscosity-type arguments. 

\smallskip

The uniform $L^\infty$-bound follows by backward induction of  a straightforward application of the comparison principle, which implies that, for any $n\in \{0, \dots, N-1\}$,  
$$
\sup_{t\in [t^N_n,t^N_{n+1})} \|\tilde u^N_t\|_\infty\leq \|\tilde u^{N}_{t^{N,-}_{n+1}}\|_\infty +CN^{-1} \leq \|\tilde u^N_{t^{N,+}_{n+1}}\|_\infty +CN^{-1}, 
$$
where the last inequality holds because the conditional expectation is a contraction in the $L^\infty$-norm. 
\smallskip
%
\smallskip

For the semiconcavity estimate, we note that, if $v$ is a viscosity solution of the Hamilton-Jacobi equation 
$$
-\partial_t v +H_t(Dv,x)=0 \ \ {\rm in} \ \  \R^d\times (0,T) \quad  v_T(\cdot)=  v_T \ \ \text{in} \ \ \R^d, 
$$
with  $H$ satisfying \eqref{HH},  
then there exist $c_0, c_1>0$ such that, for any $|z|\leq 1$ and in the sense of distributions, 
\be\label{kjhenrdkfg}
\text{if} \ \ D^2v_T \  z\cdot z-\lambda v_T \leq c_1, \ \ \text{then} \ \  D^2v_t \ z\cdot z-\lambda v_t\leq c_1+c_0T. 
\ee
Indeed, for any $z\in \R^d$ with $|z|\leq 1$, the map $w_t(x)= D^2v_t(x)z\cdot z-\lambda v_t(x)$ satisfies  (formally) 
\begin{align*}
& -\partial_t w_t+ D_pH_t(Dv_t,x)\cdot Dw_t +  D^2_{pp}H_t(Dv_t,x) \ Dv_{t,z}\cdot Dv_{t,z} + 2D^2_{pz}H_t(Dv_t,x)\cdot Dv_{t,z}\\[2mm]
& \qquad + D^2_{zz}H_t(Dv_t,x) -\lambda (H_t(Dv_t,x)-D_pH_t(Dv_t,x)\cdot Dv_t)=0, 
\end{align*}
and, hence, in view of given \eqref{HH}, \eqref{kjhenrdkfg} follows from the comparison principle.
\smallskip

Applying \eqref{kjhenrdkfg} to $\tilde u^N$ provides the uniform semiconcavity estimate by a backward induction argument similar to the one for the $L^\infty$-bound. The bounds above immediately imply  the Lipschitz estimate of $D\tilde u^N$. 
\smallskip

Recall that, for each $x\in \R^d$, $(\tilde u^N(x))_{t\in[0,T]}$ is adapted to the filtration $(\mathcal F_t)_{t\in[0,T]}$ and $(\tilde M_t^N(x))_{t\in [0,T]}$ is a c\`adl\`ag martingale with respect to  the discrete filtration $(\mathcal F_{t_n})$. 
\smallskip

The bound  on $\tilde M^N$ follows from the observation that, since $\tilde M_T=0$, by induction we have, for $x\in \R^d$ and $\P-$a.s.,
$$
\tilde M^N_t(x) = \tilde  u^N_t(x)-\tilde  G (x)  + \int_t^T \tilde H^N_s(D\tilde u^N_s(x),x)ds.
$$
\end{proof}

\subsection{Comparison and uniqueness}\label{subseq.Comp}

We say that $(\tilde u^1,\tilde M^1)$ (resp. $(\tilde u^2,\tilde M^2)$) is a supersolution (resp. subsolution) of  \eqref{BSDHJE_Intro},  if $\tilde u^1$ (resp. $\tilde u^2$) satisfies all the conditions of Definition \ref{def.solHJ} but \eqref{def.eq} which is replaced by the requirement that,  $\P-$a.s.,  and for a.e $(x,t,t')\in \R^d\times [0,T]\times [0,T]$ with $t<t'$,
$$
\tilde  u^1_t(x) \geq \tilde u^1_{t'}(x) -\int_t^{t'} \tilde H_s(D\tilde u^1_s(x),x)ds -\tilde M^1_{t'}(x)+ \tilde M^1_t(x) \ \ \text{and} \ \  u^1_T\geq \tilde G \ \ \text{in} \ \ \R^d,\\
$$
$\Big(\text{resp.}$
$$
\tilde  u^2_t(x) \leq \tilde u^2_{t'}(x) -\int_t^{t'} \tilde H_s(D\tilde u^2_s(x),x)ds -\tilde M^2_{t'}(x)+ \tilde M^2_t(x) \ \ \text{and} \ \   u^2_T\leq \tilde G  \ \ \text{in} \ \ \R^d. \Big)
$$
The comparison result between supersolutions and subsolutions  is stated next.
\begin{prop}[Comparison]\label{prop.comparison} Assume  \eqref{H1}, \eqref{H2}, \eqref{H2} and \eqref{H3}, and 
let $(\tilde u^1,\tilde M^1)$ and $(\tilde u^2,\tilde M^2)$ be   respectively a supersolution and a subsolution of  \eqref{BSDHJE_Intro}. Then, $\P-$a.s., 
$\tilde u^1\geq \tilde u^2$ in $\R^d\times [0,T].$ 
\end{prop}

The following uniqueness  result follows immediately. 
\begin{cor}[Uniqueness]
Assume  \eqref{H1}, \eqref{H2} and \eqref{H3}.  Then  there exists a unique solution to \eqref{BSDHJE_Intro}.
\end{cor}

\begin{proof}[Proof of Proposition~\ref{prop.comparison}]  The proof follows  again Douglis's uniqueness proof and is very similar with the proof of Theorem~\ref{thm.mainexists} with $\tilde u^1$ and $\tilde u^2$ in place of $\tilde u^N$ and $\tilde u^K$. Hence, in what follows we present a brief sketch.
\smallskip

Fix a smooth, convex and nonincreasing map $\phi:\R\to \R^+$, and  let $w_t(x)= \phi(\tilde u^1_t(x)-\tilde u^2_t(x))$. Then by It\^o's formula and the inequalities  satisfied by the $\tilde u^i$'s, we have 
\begin{align*}
dw_t(x) & \geq \phi' (\tilde H_t(D\tilde u^1_t(x),x)-\tilde H_t(D\tilde u^2_t(x),x)) dt +\frac12 \phi'' d<\tilde M(x)>_t+ \phi' d\tilde M_t(x) \\
& \geq  {\bf b}_t(x)\cdot D w_t(x)dt+ \phi' d\tilde M_t(x), 
\end{align*}
where $\phi$ and its derivatives are evaluated at $\tilde u^1_t(x)-\tilde u^2_t(x)$,  $\tilde M_t= \tilde M^1_t-\tilde M^2_t$
and 
$$
{\bf b}_t(x)= \int_0^1 D_p\tilde H_t((1-s)D\tilde u^1_t(x)+sD\tilde u^2_t(x))ds. 
$$ 
The rest of the proof follows almost verbatim the arguments of the proof of Theorem~\ref{thm.mainexists}. It consists of an  appropriate regularization ${\bf b}^\ep$ similar to the one in the aforementioned proof and a rewriting of the inequality satisfied by  $w_t$ as 
\begin{align*}
dw_t(x) &  \geq ( {\rm div}({\bf b}^\ep_t(x) w_t) -{\rm div}({\bf b}^\ep_t(x)) w_t +({\bf b}_t-{\bf b}^\ep_t)\cdot Dw_t )dt + \phi' d\tilde M_t.
\end{align*}
Next we consider the quantity 
$$
e_t = \E\left[ \int_{B_{\alpha+\beta t}} w_t(x)dx,\right] 
$$
and we find, as in the proof of Theorem  \ref{thm.mainexists}, for $t_1 \in [0,T]$, 
\begin{align*}
e_T-e_{t_1} &  \geq  \E\Bigl[  \int_{t_1}^T \int_{B_{\alpha+\beta t}}  ({\rm div}({\bf b}^\ep_t(x) w_t) -{\rm div}({\bf b}^\ep_t(x)) w_t +({\bf b}_t-{\bf b}^\ep_t)\cdot Dw_t ) dydt  \\
& \qquad + \beta \int_{t_1}^T \int_{\partial B_{\alpha+\beta t}} w_t(y)dy dt\Bigr] \\
& = \E\Bigl[  \int_{t_1}^T \int_{\partial B_{\alpha+\beta t}}  ({\bf b}^\ep_t(x)\cdot \nu_y + \beta) w_t dydt \\ 
 & \qquad   +\int_{t_1}^T \int_{B_{\alpha+\beta t}}( 
 -{\rm div}({\bf b}^\ep_t(x)) w_t +({\bf b}_t-{\bf b}^\ep_t)\cdot Dw_t ) dydt \Bigr] . 
\end{align*}

The properties of  ${\bf b}^\ep$, a suitable choice of $\beta$ and Grownwall's inequality lead after letting $\ep\to$ to
$$
e_{t_1}\leq e^{CT} e_T.
$$
We choose (after approximation) $\phi(r)= (-r)_+$. Then $e_T=0$ since $\tilde u^1_T\geq \tilde u^2_T$. Therefore $e_t=0$ for any $t$, which shows that  $\tilde u^1\geq \tilde u^2$ since $\alpha$ is arbitrary. 

\end{proof}


\subsection{Optimal control representation} \label{subseq.OCrep} We develop a stochastic optimal control formulation for $\tilde u$ and present  a stochastic maximum principle-type result. 
\smallskip

In what follows, $\tilde L$ is the Legendre transform of $\tilde H$, that is, for $x,\alpha \in \R^d, t\in [0,T]$ and $\omega \in \Omega$, 
$$
\tilde L_t(\alpha, x,\omega)=\sup_{p\in \R^d} [ -p\cdot \alpha - \tilde H_t(p,x,\omega)].
$$
Moreover, for $x\in \R^d$ and $t \in [0,T]$, $\mathcal A_{t, x}$ is the set of admissible paths defined by 
$$
\mathcal A_{t,x}=\{\gamma \in \mathcal S^2(\R^d): \; \gamma_t=x\; \text{and}\; \gamma \in H^1([t,T];\R^d) \; \text{a.s.} \}.
$$
\begin{prop}\label{prop.repsol} Assume \eqref{H1}, \eqref{H2}, \eqref{H2} and \eqref{H3}, and let $\tilde u$ be the solution of \eqref{BSDHJE_Intro}. Then 
\be\label{eq.opticontrolrep}
\tilde u_t(x) =\underset{\gamma \in {\mathcal A_{t,x}}}{\rm essinf}
\;  \E\left[ \int_t^T \tilde L_s(\dot \gamma_s,\gamma_s)ds +\tilde G(\gamma_T)\ |\ {\mathcal F}_t\right].
\ee
\end{prop}

\begin{proof} To simplify the notation, we present the proof for $t=0$.  Let $(\tilde u^N)_{N\geq 1}$ be  as in the proof of Theorem \ref{thm.mainexists}. Since  on each time interval $(t^N_n,t^N_{n+1})$,  $\tilde u^N$ is a viscosity solution of a standard HJ equation, for any fixed $\omega$ and $\tilde L^N$ the Legendre transform of $\tilde H^N$,  we have 
$$
\tilde u^N_{t^N_n}(x)=\inf_{\gamma \in  \tilde{ \mathcal A}_{t^N_n, t^N_{n+1},x}} 
  \int_{t^N_n}^{t^N_{n+1}} \tilde L^N_s(\dot \gamma_s, \gamma_s)ds + \tilde u^N_{t^{N,-}_{n+1}}(\gamma_{t^N_{n+1}}),
$$
where $$\tilde{ \mathcal A}_{t, s, x}= \{\gamma \in H^1([t,s];\R^d): \   \gamma_t=x \}.$$
The   ${\mathcal F}_{t^N_n}-$measurability of $\tilde u^N_{t_{n+1}^-}$ allows to find a ${\mathcal F}_{t_n}-$measurable selection 

$(\omega,x)\to (\tilde \gamma^N_t (\omega,x))_{t\in [t^N_n,t^N_{n+1}]}$ of minimizers. Note that, since $D\tilde u^N$ is uniformly bounded, $\dot{\tilde \gamma}^N$ is uniformly bounded as well by some constant $C$. 
\smallskip

Concatenating these minimizers, we find, for any $x\in \R^d$, a $({\mathcal F}_{t^N_n})_{n=0,\ldots,N}-$adapted  path $\tilde \gamma^N  \in \mathcal A_{0,x}$  
such that 
$$
\tilde u^N_0(x)= \E\left[ \int_0^T \tilde L^N_s(\dot{\tilde \gamma}^N_s,\tilde  \gamma^N_s)ds + \tilde G(\tilde \gamma^N_T)\right] =\inf_{\gamma \in \mathcal A_{0,x}} \E\left[ \int_0^T \tilde L^N_s(\dot \gamma_s, \gamma_s)ds + \tilde G(\gamma_T)\right],
$$
where the second equality can be proved by dynamic programming and induction. 
\smallskip

Then, in view of  the continuity of the filtration $({\mathcal F}_t)_{t\in [0,T]}$ and the definition of $\tilde L^N$, we find 
$$
 \tilde u_0(x) = \lim_{N\to \infty} u^N_0(x)\leq \underset{\gamma \in \mathcal A_{0,x}}\inf \E\left[ \int_0^T \tilde L_s(\dot \gamma_s, \gamma_s)ds + \tilde G(\gamma_T)\right].
$$
On the other hand, the time regularity of $\tilde L$ and the uniform in $N$ bound  on  $\dot{\tilde \gamma}^N$, which we denote by $C$, imply that 
$$
 \tilde u^N_0(x) \geq  
 \E\left[ \int_0^T \tilde L_s(\dot{\tilde \gamma}^N_s,\tilde  \gamma^N_s)ds + \tilde G(\tilde \gamma^N_T)\right]-\E[\omega^N_C]. 
 $$
It follows  that 
 \begin{align*}
 \tilde u_0(x) & = \lim_{N\to \infty} u^N_0(x) \geq \limsup_{N\to \infty}  \Big[\underset{\gamma \in \mathcal A_{0,x}} \inf \ \E\left[ \int_0^T \tilde L_s(\dot \gamma_s, \gamma_s)ds + \tilde G(\gamma_T)\right] -\E[\omega^N_C] \Big]\\[2mm]
 &= \underset{\gamma \in \mathcal A_{0,x}} \inf \E\left[ \int_0^T \tilde L_s(\dot \gamma_s, \gamma_s)ds + \tilde G(\gamma_T)\right].
 \end{align*}
\end{proof}

We now discuss the maximum principle and the regularity of the value function along optimal solutions. We point, however, that we do not claim the existence of an optimal solution. 

\begin{thm}[Maximum principle]\label{thm.MaxPple}  Assume \eqref{H1}, \eqref{H2}, \eqref{H2} and \eqref{H3}, let $\bar \gamma \in \mathcal A_{0,x}$ be optimal for $\tilde u_0(x)$ and define, for $t\in [0,T]$,  the  $({\mathcal F}_t)_{t\in [0,T]}-$adapted continuous process $\bar p$ by 
\be\label{defbarpt}
\bar p_t =\E\left[  \int_t^T D_x\tilde L(\dot{\bar \gamma}_s, \bar \gamma_s)ds +D\tilde G(\bar \gamma_T)\ \Big| \ {\mathcal F}_t\right].
\ee
Then,  for $t\in [0,T]$,
\be\label{eqEuler}
\dot{\bar \gamma}_t= -D_p\tilde H_t(\bar p_t, \bar \gamma_t)
\ee
and $\bar p$ solves the BSDE
\be\label{eqEulerq}
d\bar p_t = D_x\tilde H_t(\bar p_t, \bar \gamma_t)dt + d \overline m_t  \ \ \text{in} \ \  [0,T], \quad \bar p_T= D\tilde G(\bar \gamma_T),
\ee
where $( \overline m_t)_{t\in [0,T]}\in \mathcal S^r(\R)$ is a continuous martingale. 
\end{thm}

\begin{rmk}{\rm The theorem implies that $\bar \gamma$ is of class $C^1$.}\end{rmk}

\begin{proof} Fix $h>0$  small and $t\in [0,T)$  at which $\bar \gamma$ is differentiable $\P-$a.s.,  and let  $v\in L^\infty(\Omega, \R^d)$ be ${\mathcal F}_t$-measuarable. 
\smallskip

Define $\gamma^h$ by $\gamma^h(0)=x$ and 
$$
\dot\gamma^h_s = \left\{\begin{array}{ll}
\dot{\bar \gamma}_s & {\rm if }\; s\in [0,t]\cup [t+h, T],\\[2mm]
v & {\rm otherwise},
\end{array}\right.
$$
note that, since for $s\geq t+h$,  
\be\label{helskjrndc}
\gamma^h_s = \bar \gamma_s + hv-(\bar \gamma_{t+h}-\bar \gamma_t), 
\ee
$\gamma^h$ is admissible, and use the dynamic programming principle to get 
\begin{align*}
 & \E\left[ \int_t^T \tilde L_s(\dot \gamma^h_s,\gamma^h_s)ds +\tilde G(\gamma^h_T)\ |\ {\mathcal F}_t\right]
 \geq  \E\left[ \int_t^T \tilde L_s(\dot{\bar \gamma_s},\bar \gamma_s)ds +\tilde G(\bar \gamma_T)\ |\ {\mathcal F}_t\right].
\end{align*}
Hence 
\begin{align*}
 & \E\Bigl[ \int_t^{t+h} (\tilde L_s(v,\gamma^h_s)- \tilde L_s(\dot{\bar \gamma_s},\bar \gamma_s))ds\\
 & \qquad  +\int_{t+h}^T (\tilde L_s(\dot{\bar \gamma}_s,\gamma^h_s)
 - \tilde L_s(\dot{\bar \gamma}_s,\bar \gamma_s))ds  +\tilde G(\gamma^h_T)-\tilde G(\bar \gamma_T)\ |\ {\mathcal F}_t\Bigr]
\ \geq \; 0.
\end{align*}
Then 
\begin{align*}
 & \E\Bigl[ \int_t^{t+h} (\tilde L_s(v,\gamma^h_s)- \tilde L_s(\dot{\bar \gamma_s},\bar \gamma_s))ds\\
 & +\int_{t+h}^T (\tilde L_s(\dot{\bar \gamma}_s,\gamma^h_s)
 - \tilde L_s(\dot{\bar \gamma}_s,\bar \gamma_s))ds  + D\tilde G(\bar \gamma_T) \cdot (\gamma^h_T-\bar \gamma_T) + C |\gamma^h_T-\bar \gamma_T|^2 \ |\ {\mathcal F}_t\Bigr]
\ \geq \; 0.
\end{align*}
Dividing by $h$, letting $h\to 0^+$ and using    \eqref{helskjrndc} we find 
\begin{align*}
 & \tilde L_t(v,\bar \gamma_t)- \tilde L_t(\bar \gamma_t,\bar \gamma_t) + 
(v- \dot{\bar \gamma}_t)\cdot \E\left[ \int_{t}^T D_x\tilde L_s(\dot{\bar \gamma}_s,\bar \gamma_s)ds  +D\tilde G(\bar \gamma_T)\ |\ {\mathcal F}_t\right]\geq 0.
\end{align*}
Since $v\in {\mathcal F}_t$ is arbitrary, we conclude that, $\P-$a.s., 
$$
D_\alpha \tilde L_t(\bar \gamma_t,\bar \gamma_t) + \E\left[ \int_{t}^T D_x\tilde L_s(\dot{\bar \gamma}_s,\bar \gamma_s)ds  +D\tilde G(\bar \gamma_T)\ |\ {\mathcal F}_t\right]=0,
$$
and  \eqref{eqEuler} holds with $\bar p_t$ is defined by \eqref{defbarpt}. 
\smallskip

To prove \eqref{eqEulerq}, we first note that the (standard) BSDE \eqref{eqEulerq} has a unique solution and $D_x\tilde L_t(\alpha, x)=- D_x\tilde H_t(p,x)$ if $\alpha=-D_p\tilde H_t(p, x)$. Thus, in view of \eqref{eqEuler}, we have $D_x\tilde L_t(\dot{\bar \gamma}_t, \bar \gamma_t)= - D_p\tilde H_t(\bar p_t, \bar \gamma_t)$ and \eqref{defbarpt} can be written, for $t\in [0,T]$, as 
$$
\bar p_t =\E\left[  - \int_t^T D_x\tilde H(\bar p_s, \bar \gamma_s)ds +D\tilde G(\bar \gamma_T)\ \Big| \ {\mathcal F}_t\right]. 
$$
It follows  that $t\to \bar p_t-\bar p_0-\int_0^t D_x\tilde H(\bar p_s, \bar \gamma_s)ds$ is a martingale, which proves \eqref{eqEulerq}. 

\end{proof}

The next result is about the regularity of $\tilde u$ along the optimal path.

\begin{lem}\label{lem.mt=Du} Let $\bar \gamma$ and  $\bar p$ be as in Theorem~\ref{thm.MaxPple}. Then, $\P-$a.s. and  for any $t\in (0,T]$, $x\to \tilde u_t(x)$ is differentiable at $\bar \gamma_t$ and $\bar p_t =D\tilde u_t(\bar \gamma_t)$. 
\end{lem} 

\begin{proof} It follows from the dynamic programming principle that, for any $h>0$ small, all   ${\mathcal F}_{t}$ measurable and bounded $v$, and $\gamma^h$ such that $\gamma^h_{t-h}= \bar \gamma_{t-h}$ and $\dot \gamma^h_s= v^h= \E[v\ |\ {\mathcal F}_{t-h}]$ on $[t-h,t]$,  
\[
\tilde u_{t-h}(\bar \gamma_{t-h}) = \E\bigl[ \int_{t-h}^t \tilde L_s(\dot{\bar \gamma}_s, \bar \gamma_s)ds+ \tilde u_t(\bar \gamma_t)\ \Bigl| \; {\mathcal F}_{t-h}\Bigr]  \leq \E\bigl[ \int_{t-h}^t \tilde L_s(v^h, \gamma^h_s)ds+ \tilde u_t(\gamma^h_t)\ \Bigl| \; {\mathcal F}_{t-h}\Bigr].\] 
Let $q=q(\omega)$ be a measurable selection of $D^+u_t(\bar \gamma_t,\omega)$. Then, using the semiconcavity of $\tilde u_t$, we find 
\begin{align*}
0 & \leq \E\bigl[ \int_{t-h}^t (\tilde L_s(v^h, \gamma^h_s)-\tilde L_s(\dot{\bar \gamma}_s, \bar \gamma_s)) ds+ q\cdot (
\gamma^h_t- \bar \gamma_t)+C|
\gamma^h_t- \bar \gamma_t|^2\ \Bigl| \; {\mathcal F}_{t-h}\Bigr].
\end{align*}
Dividing  by $h$, letting  $h\to 0^+$ and using that the filtration $({\mathcal F}_t)$ is continuous, we obtain  
\begin{align*}
0 & \leq \tilde L_t(v, \bar \gamma_t)-\tilde L_t(\dot{\bar \gamma}_t, \bar \gamma_t)+ q\cdot (v-\dot{ \bar \gamma}_t).  
\end{align*}
Since  $\dot{\bar \gamma}_t$ maximizes $v\to -q\cdot v- \tilde L_t(v, \bar \gamma_t)$, it follows from  \eqref{eqEuler},  that 
 $$
\dot{\bar \gamma}_t= - D_p\tilde H_t(q, \bar \gamma_t)= -D_p\tilde H_t(\bar p_t, \bar \gamma_t).
$$
Thus  $q= \bar p_t$ and $D^+\tilde u_t(\gamma_t)$ is  a singleton. In view of the semiconcavity of $\tilde u_t$, the last fact implies that $\tilde u_t$ is differentiable at $\bar \gamma_t$ and  $\bar p_t =D\tilde u_t(\bar \gamma_t)$.
\end{proof}
\subsection{The continuity equation}\label{subseq.coneq}

We now investigate the continuity equation associated with the vector field $ -D_p\tilde H(D\tilde u_t(x),x)$. As in the previous subsections, $(\tilde u,\tilde M)$  is the solution of \eqref{BSDHJE_Intro}. 

\begin{prop}\label{prop.uniConti} Assume \eqref{H1}, \eqref{H2} and \eqref{H3}.  Then, for each  $\bar m_0\in L^\infty(\R^d)\cap \mathcal P_2(\R^d)$, there exists a $({\mathcal F}_t)_{t\in [0,T]}-$adapted process $\tilde m\in \mathcal S^r({\mathcal P}_1(\R^d))\cap L^\infty(\R^d\times (0,T))$ which solves, $\P-$a.s. and  in the sense of distributions,  the continuity equation (with random coefficients) 
\be\label{eq.eqmgivenDxu0}\partial_t\tilde m_t =  {\rm div}(\tilde m_tD_p\tilde H_t(D\tilde u_t(x),x))dt\ \  {\rm in}\ \ \R^d \times (0,T) \  \ \ 
 \tilde m_0=\bar m_0 \ \  {\rm in} \ \ \R^d. 
\ee
\end{prop}

\begin{proof} 
We use the discretization  in the proof of Theorem \ref{thm.mainexists} and  consider   the solution $(\tilde u^N, \tilde M^N)$ of the discretized problem defined there. 
\smallskip

Let $\tilde m^N \in C([0,T]; {\mathcal P}_1(\R^d))$ be the $(\mathcal F_{t_n})_{n=0,\dots,N}-$adapted process that solves, in the sense of distributions,  
\be\label{eq.mN}
\partial_t \tilde m^N_t =  {\rm div}(\tilde m^N_tD_p\tilde H^N_t(D\tilde u^N_t,x))dt\ \ {\rm in} \ \ \ \R^d\times (0,T) \ \ \  
\tilde m_0=\bar m_0 \ \ \ {\rm in} \ \ \R^d,
\ee
which,  following  the discussion in the appendix of \cite{CaHa}, can be built by induction. Indeed,  we can construct $\tilde m^N$ on each time interval $[t^N_n, t^N_{n+1}]$, since on this interval $\tilde u^N$ satisfies a standard HJ equation. In addition, for some $C>0$ and  $\P-$a.s.,
$$
\begin{array}{lll} 
&(i) & \ds  \sup_{t\in [0,T]} \int_{\R^d} |x|^2 \tilde m^N_t(x)dx\leq C,\\[3mm]
&(ii) &\ds  {\bf d}_1(\tilde m^N_t, \tilde m^N_s)\leq C|s-t| \ \  \text{for all} \ \  s,t\in [0,T],\\[2.5mm]
& (iii)&\ds  \|\tilde m^N\|_\infty \leq C. 
\end{array}
$$
Let $\tilde m$ be an (up to a subsequence) limit of $(\tilde m^N)_{N\in \N}$ in $L^\infty(\R^d\times [0,T]\times \Omega)-$weak*. Then, since, as $N\to \infty$ and  $\P-$a.s. and $(x,t)$ a.e., $D\tilde u^N \to D\tilde u$,  we can pass to the limit in  \eqref{eq.mN}. The claim then follows. 
\end{proof}

We now turn to the question of uniqueness, for which, unfortunately, we  require a much stronger condition than the standing ones. Indeed, we need to assume that $\tilde H$ is of the form
\be\label{hypsup}
\tilde H_t(p,x)= \frac{\tilde a_t(x)}{2} |p|^2 + \tilde B_t(x)\cdot p +\tilde f_t(x), 
\ee
where, for some constant $C_0>0$,   
\be\label{takis100}
\tilde a, \tilde f \in \mathcal S^2(C^2(\R^d)) \ \text{and}\ \tilde B\in \mathcal S^2(C^2(\R^d; \R^d)) \ \ \text{ with} \ C_0^{-1}\leq \tilde a_t(x)\leq C_0.
\ee
We note that \eqref{hypsup} and \eqref{takis100} yield that  $\tilde H$ satisfies \eqref{H1},  \eqref{H2},  and \eqref{H3}. 
\smallskip

The following result is a variation of  the one in \cite{BeLaLi20}. 

\begin{prop}\label{prop.UnikCont} Assume \eqref{H1}, \eqref{hypsup} and \eqref{takis100}.
Then, for each  $\bar m_0\in L^\infty(\R^d)\cap \mathcal P_1(\R^d)$, there exists a unique  $({\mathcal F}_t)_{t\in [0,T]}-$adapted process $\tilde m\in \mathcal S^2({\mathcal P}_1(\R^d))$ with bounded density in $\R^d\times (0,T))$ which solves \eqref{eq.mN}, $\P-$a.s. and  in the sense of distributions.
\end{prop}

In the deterministic case considered in the appendix of \cite{CaHa}, the uniqueness of a solution does not require an addition structure assumption on $\tilde H$. Instead, it  relies on the fact that the forward-forward Hamiltonian system \eqref{eqEuler}-\eqref{eqEulerq} has a unique solution given the initial condition $(\bar \gamma_0,\bar p_0)$. Unfortunately  this does not seem to be  the case in the random setting.

\begin{proof}[Proof of Proposition~\ref{prop.UnikCont}] Set $\tilde b_t(x)= -D_p\tilde H_t(D\tilde u_t(x),x)$.  Then, in view of \eqref{hypsup} and the definition of a solution $\tilde u$ of  \eqref{BSDHJE_Intro}, $\tilde b_t(x)=-\tilde a_t(x) D\tilde u_t(x)-\tilde B_t(x)$ is bounded and one-side Lipschitz, that is, there exists $C_0>0$ such that, 
for all  $(x,y,t)\in  \R^d\times \R^d \times [0,T],$
$$ 
\|\tilde b\|_\infty\leq C_0  \ \ \text{and} \ \ (\tilde b_t(x)-\tilde b_t(y))\cdot(x-y) \geq -C_0|x-y|^2. 
$$
Then we can apply $\omega$ by $\omega$ the uniqueness result to the continuity equation given in Proposition \ref{prop.appen} in the appendix. 
\end{proof}

\subsection{Existence of optimal paths of  the stochastic control problem}\label{subsec.exOS}

We now address the problem of the existence of optimal paths  for  the  control representation of $\tilde u$ established in Proposition~\ref{prop.repsol} . For simplicity we assume again that $t=0$ and recall that 

\be\label{eq:u0=inf}
\tilde u_0(x) ={\rm essinf}_{\gamma \in \mathcal A_{0,x}}  
 \E\left[ \int_t^T \tilde L_s(\dot \gamma_s,\gamma_s)ds +\tilde G(\gamma_T)\right].
\ee
The problem is that   \eqref{eq:u0=inf} is a non convex stochastic optimal control problem with a (a priori) non smooth value function, hence the existence of an optimal path is far from obvious. 

\begin{prop}\label{prop.ExistOptTraj} Assume \eqref{H1},  \eqref{hypsup} and \eqref{takis100}. Then, $\P-$a.s. and  for a.e. $x\in \R^d$ there exists a unique minimizer $\bar \gamma^x \in \mathcal A_{0,x}$ of the stochastic optimal control problem \eqref{eq:u0=inf} and this minimizer satisfies  $\dot{\bar \gamma}^x_t= -D_p\tilde H_t(D\tilde u_t(\bar \gamma_t), \bar \gamma_t)$. 
\end{prop}

\begin{proof} Fix  $\bar m_0\in \mathcal P_2(\R^d)$ with  smooth and positive density. According to Proposition \ref{prop.UnikCont},  the random continuity equation 
$$
\partial_t \tilde m_t -{\rm div}(\tilde m_t D_p\tilde H_t(D\tilde u_t,x))= 0 \ \ \text{in} \ \  \R^d\times (0,\infty) \ \ \  \tilde m_0=\bar m_0 \ \ \text{in} \ \ \R^d
$$
has a unique solution $\tilde m$ with  bounded density. 
\smallskip

A simple adaptation of the  Lagrangian approach introduced in Ambrosio \cite{Am04} shows  that, in view of  the uniqueness of the solution to the continuity equation, there exists a  $dx\times d\P-$a.e. unique Borel measurable map $\R^d\times \Omega \ni (x,\omega)\to \bar \gamma^{x,\omega} \in \Gamma=C([0,T];\R^d)$,
which is adapted to the filtration $(\mathcal F_t)_{t\in [0,T]}$, such that, $\P$-a.s. and for any $t\in [0,T]$,  $\bar \gamma^{x,\omega}_0=x$ and 
\vskip-.1in
$$
\tilde m_t(\cdot, \omega)= \int_{\R^d} \delta_{\bar \gamma^{x,\omega}_t} \bar m_0(x)dx.
$$
\vskip-.1in

We show next  that the process $(\bar \gamma^x_t)_{t\in [0,T]}$ is optimal in the optimization problem \eqref{eq:u0=inf} for $\bar m_0-$a.e. $x\in \R^d$. For this, we claim that 
\be\label{alkjensdf}
\int_{\R^d} \E\left[ \int_0^T \tilde L_t(\dot{\bar \gamma}^x_t, \bar \gamma^x_t) dt+\tilde G(\bar \gamma^x_T) \right] m_0(x)dx= \int_{\R^d} \tilde u_0(x)m_0(x)dx. 
\ee
Assuming for the moment  \eqref{alkjensdf}, we proceed with the proof of the optimality, and  recall that, for any $x\in \R^d$, 
$$
\tilde u_0(x)\leq \E\left[ \int_0^T \tilde L_t(\dot{\bar \gamma}^x_t, \bar \gamma^x_t) dt+\tilde G(\bar \gamma^x_T) \right] .
$$
Integrating  the inequality above against $\bar m_0$, we infer by \eqref{alkjensdf} that $\bar \gamma^x$ is optimal for $\bar m_0-$a.e. $x\in \R^d$. Since  $\bar m_0>0$, this holds for a.e. $x\in \R^d$. 
\smallskip

It remains to prove \eqref{alkjensdf}. For this let, $t_n=nT/N$ for $n\in\{0,\dots, N\}$ with  $N\in \N$  and note that, in view of  the equation satisfied by $\tilde u$, we have, for a.e. $x$, 
\be\label{takis121}
 \tilde u_{t_{n+1}}(x)  -\tilde u_{t_{n}}(x) =  - \int_{t_n}^{t_{n+1}} \tilde H_t(D\tilde u_t(x),x) dt - (\tilde M_{t_{n+1}}(x)-\tilde M_{t_n}(x)).
\ee

Integrating \eqref{takis121}  against $\tilde m_{t_n}$, which is absolutely continuous with a bounded density, and summing over $n$ gives
\begin{align*}
& \sum_{n=0}^{N-1}  \int_{\R^d} (\tilde u_{t_{n+1}}(x)- \tilde u_{t_{n}}(x))\tilde m_{t_n}(x)dx \\
& \qquad = - \sum_{n=0}^{N-1}  \int_{t_n}^{t_{n+1}}\int_{\R^d} \tilde H_t(D\tilde u_t(x),x)\tilde m_{t_n}(x)dxdt 
-\sum_{n=0}^{N-1}  \int_{\R^d} (\tilde M_{t_{n+1}}(x)-\tilde M_{t_n}(x)) \tilde m_{t_n}(x)dx . 
\end{align*}
Reorganizing  the left-hand side of the expression above  taking into account the equation satisfied by $\tilde m$ yields 
\begin{align*}
& \sum_{n=0}^{N-1}  \int_{\R^d} (\tilde u_{t_{n+1}}(x)- \tilde u_{t_{n}}(x))\tilde m_{t_n}(x)dx \\
&  = \int_{\R^d} \tilde u_{T}(x)\tilde m_{t_{N-1}}(x)dx-  \int_{\R^d} \tilde u_0(x)\tilde m_0(x)dx-  \sum_{n=1}^{N-1}  \int_{\R^d} \tilde u_{t_{n}}(x)(\tilde m_{t_n}(x)-\tilde m_{t_{n-1}}(x))dx \\
&  = \int_{\R^d} \tilde G(x)\tilde m_{t_{N-1}}(x)dx -  \int_{\R^d} \tilde u_0(x)\bar m_0(x)dx
\\
& \qquad 
+ \sum_{n=1}^{N-1}\int_{t_{n-1}}^{t_n}  \int_{\R^d} D\tilde u_{t_{n}}(x)\cdot D_p \tilde H_{t}(D\tilde u_t(x),x)\tilde m_t(x)dxdt. 
\end{align*}
We let $N\to+\infty$ and take expectation to find 
\begin{align*}
& \E\left[ \int_{\R^d} \tilde G(x) \tilde m_T(x)dx -  \int_{\R^d} \tilde u_0(x)\bar m_0(x)dx
+\int_0^T \int_{\R^d} D\tilde u_{t}(x)\cdot D_p \tilde H_{t}(D\tilde u_t(x),x)\tilde m_t(x)\right]\\
& =  \E\left[-\int_{0}^{T}\int_{\R^d} \tilde H_t(D\tilde u(x)_t,x)\tilde m_{t}(x)dxdt  
 -\int_0^T  \int_{\R^d} \tilde m_{t}(x)d\tilde M_{t}(x)dx \right]. 
\end{align*}

Recalling that $\tilde M$ is a martingale and that $p\cdot D_p\tilde H_t(p,x)+\tilde H_t(p,x)= \tilde L_t(-D_p\tilde H_t(D\tilde u_t(x),x),x)$, we rearrange the last  expression to get 
\begin{align*}
& \E\left[ \int_{\R^d} \tilde G(x) \tilde m_T(x)dx +\int_0^T \int_{\R^d}  \tilde L_t(-D_p\tilde H_t(D\tilde u_t(x),x),x) \tilde m_t(x) \right]=   \int_{\R^d} \tilde u_0(x)\bar m_0(x)dx . 
\end{align*}
Finally,   the facts that $$\tilde m_t= \int_{\R^d} \delta_{\bar \gamma^x_t} \bar m_0(x)dx \ \ \text{  and} \ \  \dot{\bar \gamma}^x_t=-D_p\tilde H_t(D\tilde u_t(\bar \gamma^x_t),\bar \gamma^x_t),\bar \gamma^x_t)  \ \ \P\otimes \bar m_0- \text{a.e.} (\omega,x) $$
imply  that \eqref{alkjensdf} holds. 
\end{proof}
 
 
\section{The stochastic MFG system} 

We investigate the stochastic MFG system \eqref{e.MFGstoch}. We begin recalling that, after the change of the unknowns in \eqref{takis101} we obtain, at least formally,  \eqref{stoMFG_Intro} which we study here. 
\subsection{The assumptions and the notion of solution}\label{subsec.exTakis}
To study \eqref{stoMFG_Intro}  we assume that 
\be\label{H'1}
\tag{MFG1}
\bar m_0\in \mathcal P_2(\R^d)\cap L^\infty(\R^d),
\ee
\be\label{H'2}
\tag{MFG2}
\begin{cases}
\tilde F: \R^d\times  [0,T]\times \mathcal P_1(\R^d)\times \Omega\to \R \ \text{is such that, for any $m\in \mathcal P_1(\R^d)$,}\\[2mm]
 \ \tilde F_\cdot (\cdot, m) \in \mathcal S^r(C^2(\R^d))\  \text{for any $r\geq 1$, \  and}\\[2mm]
\text{$\tilde F_t$ satisfies \eqref{FG} uniformly in $t\in [0,T]$ and in $\omega\in \Omega$,}
\end{cases}
\ee
and 
\be\label{H'3}
\tag{MFG3}
\begin{cases}
\tilde G: \R^d\times  \mathcal P_1(\R^d)\times \Omega\to \R \ \text{ is $\mathcal F_T-$measurable, } \ G\in C(\R^d\times \mathcal P_1(\R^d);\R), \\[2mm]
 \text{ and \eqref{FG} is satisfied uniformly in $\omega\in \Omega$.}
\end{cases}
\ee

For $\tilde H:\R^d\times[0,T]\times \Omega \to \R$, we assume that
\be\label{H'4}
\tag{MFG4}
\tilde H_t \ \text{satisfies \eqref{HH} uniformly in $t\in [0,T]$ and in $\omega\in \Omega$.}
\ee
\smallskip
Moreover, if 
$$
\omega^N_R= \underset{|s-t|\leq 1/N, y\in \R^d, m\in \mathcal P_1(\R^d), |p|\leq R} \sup \ \Big[ |\tilde H_s(y,p)-\tilde H_t(y,p)|
+|\tilde F_s(y,m)-\tilde F_t(y,m)|\Big],
$$
then, for any $R>0$,  
\be\label{H'5}
\tag{MFG5}
\underset{N\to \infty} \E [\omega^N_R]=0. 
\ee
\vskip-.075in
Finally, we assume that 
\be\label{H'6}
\tag{MFG6}
\begin{cases}
\text{ $\tilde F_t$ and $\tilde G$ \ are strongly monotone uniformly in $t\in [0,T]$ and  in $\omega\in \Omega$, }\\[2mm]
\text{and $\tilde F_t$  \ is strictly monotone for all $t\in [0,T]$ and $\P-$a.s. in $\omega$.}
\end{cases}
\ee

A classical example of a map $\tilde F$ satisfying the above conditions, which  goes back to \cite{LL06cr1, LL06cr2, LLJapan, LiCoursCollege}, is  of the form 
$$
\tilde F_t(x,m) = \tilde f_t (\cdot, m\ast \rho (\cdot)) \ast \rho,
$$
where
\[\begin{cases}\text{$\rho$ is a smooth, non negative and even kernel,} \\[2mm]
\text{with Fourier transform $\hat \rho$ vanishing almost nowhere,}
\end{cases}\]

and
\smallskip

$\tilde f:  \R^d \times [0,T]\times  \R\times \Omega\to \R$ is $(\mathcal F_t)_{t \in [0,T]}$ adapted and, for any $R>0$,
\[
\begin{cases} \text{there exists $C_R>0$ such that}\\[2mm]
\underset{0\leq s\leq R,  \; t\in [0,T]}\sup\left[
\|\tilde f_t(\cdot, s)\|_\infty +\|D\tilde f_t(\cdot ,s)\|_\infty + \|D^2\tilde f_t(\cdot ,s)\|_\infty\right] \leq C_R, \end{cases}\]
and

\[\underset{N\to \infty} \E [ \underset{|t_1-t_2|\leq 1/N, 0\leq s \leq R} \sup \ |\tilde f_{t_1}(y,s)-\tilde f_{t_2}(y,s)|]=0,\]
and, finally,
\smallskip

$\tilde f$ is strictly increasing and Lipschitz in the second variable, that is,  there exists $\alpha\in (0,1)$ such that 
\[ \alpha \leq \frac{\partial \tilde f_t}{\partial s}(x,s)\leq \alpha^{-1}.\]

%

It is immediate that  $\tilde F$ satisfies the regularity conditions in \eqref{H'2} and \eqref{H'5} and, moreover, 
\begin{align*}
& \int_{\R^d} (\tilde F_t(x,m_1)-\tilde F_t(x,m_2)) (m_1-m_2)(dx) \\
& \qquad = \int_{\R^d} (\tilde f_t(x,m_1\ast \rho(x) )-\tilde f_t(x,m_2\ast\rho(x))) (m_1\ast\rho(x)-m_2\ast\rho(x))dx\\
&\qquad  \geq \alpha \int_{\R^d} (m_1\ast \rho(x) )-m_2\ast\rho(x)))^2 dx. 
\end{align*}
Then
\begin{align*}
& \int_{\R^d} (\tilde F_t(x,m_1) -\tilde F_t(x,m_2))^2 dx \leq \|\rho\|_\infty^2 \int_{\R^d} (\tilde f_t(x,m_1\ast \rho(x)) -\tilde F_t(x,m_2\ast \rho(x)))^2 dx\\
& \qquad \leq  \|\rho\|_\infty^2 \alpha^{-2} \int_{\R^d} (m_1\ast \rho(x) -m_2\ast \rho(x))^2 dx\\
&\qquad \leq \|\rho\|_\infty^2 \alpha^{-3}\int_{\R^d} (\tilde F_t(x,m_1)-\tilde F_t(x,m_2)) (m_1-m_2)(dx),
\end{align*} 
and $\tilde F$ is strongly monotone. 
\smallskip

The strict monotonicity follows from the observation that, if
$$
\int_{\R^d} (\tilde F_t(x,m_1)-\tilde F_t(x,m_2)) (m_1-m_2)(dx) =0, 
$$
then $(m_1-m_2)\ast \rho=0$, which implies that $\widehat{(m_1-m_2)}\hat \rho=0$. Since $\hat \rho$ vanishes almost nowhere, it follows that $\widehat{(m_1-m_2)}=0$ and, hence,  $m_1=m_2$.

\smallskip
\begin{rmk}\label{manos}{\rm It follows from  \eqref{H'2} and \eqref{H'6} that $\tilde F$ is H\"{o}lder continuous in $m$ with respect to the ${\bf d}_1-$ distance, that is, for all $m_1,m_2 \in  {\mathcal P}_1(\R^d)$
$$
\|\tilde F_t(\cdot,m_1)-\tilde F_t(\cdot,m_2)\|_\infty\leq C {\bf d}_1^{1/(d+2)}(m_1,m_2).
$$
Indeed,   in view of the interpolation inequality  $$\|f\|_\infty \leq C_d \|Df\|_\infty^{1/(d+2)}\|f\|^{2/(d+2)}_{L^2},
$$
with  $C_d$ depending  only on the  dimension, we find 
\begin{align*}
& \|\tilde F_t(\cdot,m_1)-\tilde F_t(\cdot,m_2)\|_\infty^{d+2} \\[2mm]
&\qquad  \leq C_d (\|D\tilde F_t(\cdot, m_1)\|_\infty+\|D\tilde F_t(\cdot, m_2)\|_\infty) \int_{\R^d} (\tilde F_t(x,m_1)-\tilde F_t(x,m_2))^2dx\\[2mm]
&\qquad  \leq 2C_0C_d \alpha^{-1} \int_{\R^d} (\tilde F_t(x,m_1)-\tilde F_t(x,m_2)) (m_1-m_2)(dx) \\[2mm]
&\qquad  \leq  2C_0C_d \alpha^{-1}(\|D\tilde F_t(\cdot, m_1)\|_\infty+\|D\tilde F_t(\cdot, m_2)\|_\infty){\bf d}_1(m_1,m_2) \\[2mm]
& \qquad   \leq 4C_0^2C_d\alpha^{-1} {\bf d}_1(m_1,m_2). 
\end{align*}
}\end{rmk}
\smallskip

We continue  with the definition of  a weak solution of  \eqref{stoMFG_Intro}.
\begin{defn}
The triplet  $(\tilde u, \tilde m, \tilde M)$ is a solution of \eqref{stoMFG_Intro} if: 

(i)~$\tilde u\in \mathcal S^r(W^{1,1}_{loc}(\Rd))$, $\tilde M\in \mathcal S^r(L^1_{loc}(\R^d))$ and 
$\tilde m\in \mathcal S^r(\mathcal P_1(\R^d))$ for any $r\geq 1$, 

(ii)~there exits  $C>0$ such that, $\P-$a.s.,  for a.e. $t\in [0,T]$, all $z\in \R^d$ such that $|z|\leq 1$, and in the sense of distributions, 
 $$\| \tilde m\|_\infty+\|\tilde u_t\|_{W^{1,\infty}(\R^d)}+ \|\tilde M_t\|_\infty + D^2\tilde u_t \ z\cdot z  \leq C,$$
(iii)~the process $(\tilde M_t(x))$ is a $(\mathcal F_t)_{t\in [0,T]}$ continuous  martingale for a.e. $x\in \R^d$.

(iv)~for a.e $(x,t)\in \R^d\times [0,T]$ and $\P-$a.s. in $\omega$,
 \vskip-.1in
 $$
\tilde  u_t(x) = \tilde G(x, \tilde m_T) -\int_t^T ( \tilde H_s(D\tilde u_s(x),x)- \tilde F_s(x,\tilde m_s)) ds -\tilde M_T(x)+ \tilde M_t(x),
$$
 \vskip-.1in
and

(v)~in the sense of distributions and $\P-$a.s. in $\omega$, 
$$d_t\tilde m_t =  {\rm div}(\tilde m_tD_p\tilde H_t(D\tilde u_t,x))\ \  {\rm in } \ \ \R^d\times (0,T) \ \ \ 
 \tilde m_0=\bar m_0\ \  {\rm in } \ \ \R^d. $$
\end{defn}
\smallskip

\subsection{The existence and uniqueness result}\label{subsec.main}\label{subsec.exMFG}
The  main result of the paper about the existence and uniqueness of a solution of \eqref{stoMFG_Intro} is stated next.

\begin{thm}\label{thm.main} Assume \eqref{H'1}, \eqref{H'2}, \eqref{H'3}, \eqref{H'4}, \eqref{H'5} and \eqref{H'6}.  Then there exists a unique solution of \eqref{stoMFG_Intro}.
\end{thm}

The proof consists of several steps. Similarly to  the Hamilton-Jacobi case, the solution is constructed  by discretizing the noise in time. Hence, the  first step is  
to recall and refine known regularity results for deterministic MFG systems. Then we explain the construction of an approximate solution by time discretization and, finally, we pass to the limit to obtain the solution of \eqref{stoMFG_Intro}.
\vskip-.3in
\subsection{The deterministic MFG system}\label{subsec.MFGwn}

We consider the deterministic  MFG system 
\be\label{eq.MFGdeterm}
\left\{\begin{array}{l}
- \partial_t  u+ H (D u, x)=  F(x, \mu_t) \ \  {\rm in} \ \   \R^d\times (0,T), \\[2mm]
\partial_t  \mu_t -{\rm div} (  \mu_t D_p H (D u_t, x))=0  \ \  {\rm in} \ \   \R^d\times (0,T), \\[2mm]
 m_0= \bar m_0 \ \ \   u(\cdot,T)= G(\cdot, m_T). 
\end{array}\right. 
\ee
A solution of \eqref{eq.MFGdeterm} is a pair $(u,m)$ such that $u$ is a continuous, bounded and  semiconcave in $x$ uniformly in $t$  viscosity solution of  the HJ equation, while $m\in C([0,T], \mathcal P_2(\R^d))\cap L^\infty(\R^d\times (0,T))$ is a solution of the continuity equation in the sense of distribution. 
\smallskip

Next we state some general hypotheses, which imply the existence of a solution of \eqref{eq.MFGdeterm}.
\smallskip

We assume that 
\be\label{takis20}
H:\R^d\times\R^d\to \R \ \ \text{satisfies \eqref{HH}},
\ee
\be\label{takis21}
\begin{cases}
F \in C(\R^d\times \mathcal P_1(\R^d);\R) \   \text{ is Lipschitz continuous and }\\[2mm] 
\text{ semiconcave in $x$  uniformly in $m$, }
\end{cases}
\ee
\be\label{takis22}
\begin{cases}
\text{ 
there exits $\alpha_F>0$ such that, for all  $m_1,m_2\in \mathcal P_1(\R^d)$},\\[2mm] 
\int_{\R^d} ( F(x,m_1)- F(x,m_2)) (m_1-m_2)(dx) \geq \alpha_F \|F(\cdot,m_1)- F(\cdot,m_2)\|_\infty^{d+2},\\[2mm] 
\end{cases}
\ee
and
\be\label{takis23}
\begin{cases}
G:\R^d\times \mathcal P_1(\R^d)\to \R  \  \text{is H\"{o}lder continuous  in $m$ uniformly in $x$,}
\text{ bounded and}\\[2mm]
\text{ semiconcave in $x$ uniformly in $m$, and there exists $\alpha_G>0$ such that}\\[2mm] 
\text{ for all  $m_1,m_2\in \mathcal P_1(\R^d)$,}\\[2mm]
\int_{\R^d} ( G(x,m_1)- G(x,m_2)) (m_1-m_2)dx \geq \alpha_G \|G(\cdot,m_1)- G(\cdot,m_2)\|_\infty^{d+2}.
\end{cases}
\ee
\smallskip

We refer to Remark~\ref{manos} about about the connection between \eqref{takis22} and the  more standard strong monotonicity condition.

The following result can be derived from \cite{LLJapan}; see also \cite{CaHa}. In the sequel,  we give some details about the proof of the estimates that are needed for the proof of Theorem~\ref{thm.main}.  

\begin{lem}\label{lem.MFGdet} Assume \eqref{takis20}, \eqref{takis21},  \eqref{takis22} and \eqref{takis23}.
There exists $C_0>1$ such that, for any  $\bar m_0\in \mathcal P_2(\R^d)\cap L^\infty(\R^d)$, there exists a unique solution $u$ of \eqref{eq.MFGdeterm}  such that 
\be\label{takis24}
\|u\|_{\infty}\leq \|G\|_\infty+C_0T,
\ee
\be\label{takis25}
\begin{cases}
\text{ $u$ is semiconcave in $x$ uniformly in $m$, that is, if, for some $C_1>0$},\\[2mm]
\text{ all $z\in \R^d$ with  $|z|\leq 1$, all $m\in \mathcal P_2$, and in the sense of distributions,}\\[2mm]
\text{ if $D^2G(\cdot,m)z\cdot z -\lambda G(\cdot,m) \leq  C_1$, then for all $t\in [0,T]$, }\\[2mm]
\hskip1.5in D^2u_t(\cdot)z\cdot z -\lambda u_t(\cdot) \leq C_1+C_0T,
\end{cases}
\ee
and 
\be\label{takis26}
\begin{cases}
\text{ $m$ is bounded in $\R^d\times (0,T)$ and has finite second moments, that is, there}\\[2mm]
\text{exists $C>0$ depending   on  $\|Du\|_\infty$ and  the semiconcavity} \\[2mm]
\text{constant of $u$ such that, for all $t\in [0,T]$,}\\[2mm]
\hskip1.25in M_2(m_t)\leq M_2(\bar m_0)+ CT  \ \ \text{and} \ \   \|m_t\|_{\infty} \leq \|\bar m_0\|_\infty+ C T.
\end{cases}
\ee
\end{lem}

We note the claim is 
 that $u$ remains bounded and uniformly semiconcave in $x$ uniformly in time and in the initial measure $\bar m_0$.
\smallskip

In addition, $u$ is also uniformly Lipschitz continuous in $x$. This is a consequence of the elementary fact that,  
if  $v:\R^d\to \R$ is bounded by some $M$ and semiconcave with constant $K$, $v$ is Lipschitz continuous with a Lipschitz constant bounded by $2(MK)^{1/2}$. 
\smallskip

Finally,  the fact that the estimates on $u$ and $m$ grow only linearly in time $T$ will be important in the construction of the next subsection and justifies the awkward formulation of the semiconcavity estimate. 

\begin{proof}[The proof of Lemma~\ref{lem.MFGdet}]  The existence and uniqueness of the solution $(u,m)$ and the estimates on $m$  can be found in  \cite{CaHa, LLJapan}. The bound and the semiconcavity estimates on $u$ can be established as in the proof of Lemma \ref{lem.reguuNMN}. 
Here  we only repeat some the formal argument for convenience, noting that  everything  can be justified using ``viscous'' regularizations. 
\smallskip

Formally, it is immediate that 
\begin{align*}
\frac{d}{dt} M_2(m_t) & = \frac{d}{dt} \int_{\R^d} |x|^2m_t(dx) =- \int_{\R^d} 2x\cdot D_pH(Du_t(x),x)m_t(dx)\\
&  \leq 4\|D_pH(Du_t)\|_\infty^2+ M_2(m_t), 
\end{align*}
and the estimate on $M_2(m_t)$ follows by Gronwall's Lemma. 
\smallskip

For the $L^\infty-$bound, we rewrite the continuity equation as
$$
\partial_t m -D_pH(Du_t, x)\cdot Dm_t- m_t{\rm div}(D_pH(Du_t,x))=0, 
$$
where
$$
{\rm div}(D_pH(Du_t(x),x))= {\rm Tr}(D^2_{pp}H(Du_t,x)D^2u_t(x)+ D^2_{px}H(Du_t(x),x)) \leq C, 
$$
in view of the Lipschitz and semiconcavity estimates of $u$. The bound follows  using the  maximum principle. 
\end{proof}

Later in the paper  it will be convenient to define the solution of \eqref{eq.MFGdeterm} in a unique way for less regular initial measures. For this we use the following regularity result.

\begin{prop}\label{prop.MFGdet} Assume \eqref{takis20}, \eqref{takis21},  \eqref{takis22}, and \eqref{takis23}. Then,  if $(u^1,m^1), (u^2,m^2)$  are the  solutions of the MFG system \eqref{eq.MFGdeterm}  with $m^1_0,m^2_0\in \mathcal P_2(\R^d)\cap L^\infty(\R^d)$, 
then, there exist positive constants $C, C'$, which depend    on $d$, $\alpha_G$, $\alpha_F$ and  $\|Du^i\|_\infty$,  such that 
$$
\|u^1-u^2\|_\infty^{d+2} \leq C \int_{\R^d} (u^1(x,0)-u^2(x,0))(m^1_0-m^2_0)(dx) \leq C'{\bf d}_1(m^1_0,m^2_0).
$$
\end{prop}

It follows 
that the map $U:\R^d\times [0,T] \times (\mathcal P_2(\R^d)\cap L^\infty(\R^d)) \to \R$
given by 
$$
U_t(x,\bar m_0)= u_t(x), 
$$
where $(u,m)$ is the solution of the MFG system \eqref{eq.MFGdeterm} with initial condition $m(0)=\bar m_0$,  has a unique extension on $\R^d\times[0,T]\times  \mathcal P_1(\R^d)$.
\smallskip

Moreover, in view of Lemma~\ref{lem.MFGdet} and Proposition \ref{prop.MFGdet}, the extended map $U:\R^d\times [0,T]\times \mathcal P_1(\R^d)\to \R$ is  H\"{o}lder continuous in $m$  uniformly in $x$,  Lipschitz continuous and semiconcave in $x$ uniformly in $m$, and   strongly monotone in the sense of \eqref{takis22}. 
\smallskip

We also note  that the map $\bar m_0\to m$, where $(u,m)$ is the solution of the MFG system \eqref{eq.MFGdeterm} with initial condition $m(0)=\bar m_0$, is continuous in $C([0,T];\mathcal P_1(\R^d))$.  Indeed, Proposition \ref{prop.MFGdet}   gives that the map $\bar m_0\to u$ is continuous from $\mathcal P_2(\R^d)$ to $C(\R^d\times [0,T])$. In view of the semiconcavity estimate on $u$, this in turn yields the continuity of the map $\bar m_0\to Du$ in $L^1_{loc}(\R^d\times [0,T])$. The claimed  continuity of $\bar m_0\to m$ follows combining  the uniform in  time continuity of $m$ in $\mathcal P_1(\R^d)$, which depends  on $\|Du\|_\infty$, the $L^\infty-$estimate on $m$ and the uniqueness of the solution of the associated continuity equation. 

\begin{proof}[Proof of Proposition~\ref{prop.MFGdet}]  Using a viscous approximation to justify it, the standard proof of uniqueness of the MFG system \eqref{eq.MFGdeterm} yields
$$
\int_0^T \int_{\R^d} (F(x,m^1_t)-F(x,m^2_t))(m^1_t-m^2_t)(dx) \leq -\left[ \int_{\R^d} (u^1-u^2)(m^1-m^2)(dx)\right]_0^T. 
$$
It then follows from the strong monotonicity condition on $F$ and and $G$ that 
\begin{align*}
& \alpha_G \|G(\cdot,m^1_T)-G(\cdot,m^2_T)\|_\infty^{d+2} +\alpha_F \int_0^T\|F(\cdot,m^1_t)-F(\cdot,m^2_t))\|^{d+2}_\infty dt \\
&  \leq \int_{\R^d} (G(x,m^1_T)-G(x,m^2_T))(m_T^1-m_T^2)(dx) +  \int_0^T \int_{\R^d} (F(x,m^1_t)-F(x,m^2_t))(m_t^1-m_t^2)(dx)\\
&  \leq \int_{\R^d} (u^1(x,0)-u^2(x,0))(m^1_0-m^2_0)(dx) \leq (\|Du^1\|_\infty+\|Du^2\|_\infty){\bf d}_1(m^1_0,m^2_0).
\end{align*}
Using  the uniform Lipschitz estimates on $u^i$ and the comparison principle in the equations for the  $u^i$ we find 
\begin{align*}
\|u^1-u^2\|_\infty& \leq  \|G(\cdot,m^1_T)-G(\cdot,m^2_T)\|_\infty+\int_0^T \|F(\cdot,m^1_t)-F(\cdot,m^2_t))\|_\infty dt.
\end{align*}
Hence, there exists $C>0$, which depends on $d$, $\alpha_G$, $\alpha_F$ and $\|Du^i\|_\infty$, such that 
\begin{align*}
\|u^1-u^2\|_\infty^{d+2}& \leq C\Big(  \|G(\cdot,m^1_T)-G(\cdot,m^2_T)\|_\infty^{d+2}+\int_0^T \|F(\cdot,m^1_t)-F(\cdot,m^2_t)\|_\infty^{d+2} dt\Big)\\
& \leq  C \int_{\R^d} (u^1(x,0)-u^2(x,0))(m^1_0-m^2_0)(dx) \leq C{\bf d}_1(m^1_0,m^2_0).
\end{align*}
\end{proof}

\subsection{The discretized stochastic MFG system} \label{subsec.appMFG}

We use the same discretization technique and the same notation as in the proof of Theorem \ref{thm.mainexists}. 
Let $N\in \N$ and  set $t^N_n=Tn/N$,
$$
\tilde H^N_t(p,x)= \tilde H_{t^N_n}(p,x) \ \ \text{and } \ \ \tilde F^N_t (x, m)= \tilde F_{t^N_n}(x,m) \ \  {\rm on} \ \ [t^N_n,t^N_{n+1}),
$$
and consider the filtration $(\mathcal F^N_t)_{t\in [0,T]}$ defined by 
$$
\mathcal F^N_t= \mathcal F_{t^N_n} \ \  \text{for} \ \  t\in [t^N_n , t^N_{n+1}).
$$
The goal here is to build a triplet $(\tilde u^N, \tilde M^N, \tilde \mu^N)$ such that 
\smallskip

(i)~$(\tilde u^N, \tilde M^N, \tilde \mu^N)$ is adapted to the filtration $({\mathcal F}^N_{t})_{t\in [0,T]}$,
\smallskip

(ii)~on each  interval $(t^N_n, t^N_{n+1})$ with  $n=0,\ldots, N-1$,  $\tilde u^N$ is a viscosity solutions of the backward HJ equation 
$$
\left\{\begin{array}{l}
-\partial_t \tilde u^N+\tilde H^N_t (D\tilde u^N, x)= \tilde F_t^N(x,\tilde \mu^N_t) \ \ {\rm on}\ \  \R^d\times (t^N_n, t^N_{n+1}), \\[2mm]
\tilde u^N(\cdot,t^{N,-}_{n+1})= \E\left[ \tilde u^N(\cdot,t^{N,+}_{n+1}) | {\mathcal F}_{t^N_n}\right] \ \ \text{in} \ \ \R^d, 
\end{array}\right. 
$$
(iii)~$\Delta \tilde M^N$ is defined by 
$$
\Delta\tilde M^N_t(x)  = \tilde u^N(x,t^{N,-}_{n+1})- \E\left[ \tilde u^N(x,t^{N,+}_{n+1}) | {\mathcal F}_{t_n}\right]\ \  {\rm on}\ \ (t^N_n, t^N_{n+1}),\\
$$

and 
$$
\tilde M^N_t(x) = \sum_{t^N_n<t} \Delta \tilde M^N_{t^N_n}(x). 
$$

(iv)~$\tilde \mu^N \in C([0,T]; {\mathcal P}_2(\R^d))$ is $\P-$a.s. a weak solution of 
%
$$
\left\{\begin{array}{l}
\partial_t \tilde \mu^N_t -{\rm div} (\tilde  \mu^N_t D_p\tilde H^N_t (D\tilde u^N_t, x))=0 \ \  {\rm in} \ \ \R^d\times (0,T),\\[2mm]
\tilde \mu^N_0= \bar m_0 \ \  {\rm in} \ \ \R^d.
\end{array}\right. 
$$

This is the topic of the next lemma.

%
%
%
%

\begin{lem}  Assume \eqref{takis20}, \eqref{takis21}, \eqref{takis22} and \eqref{takis23}. Then, there exists at least one solution $(\tilde u^N, \tilde M^N, \tilde \mu^N)$ of the problem above and $C>0$ such that,  $\P-$a.s., a.e. in $t\in [0,T]$, all $z\in \R^d$ with $|z|\leq 1$, and in the sense of distributions, 
\smallskip

(i)~$\|\tilde u^N\|_\infty + \|D\tilde u^N\|_\infty +\|\tilde M\|_\infty + \|\tilde \mu^N\|_\infty  + D^2\tilde u^N_t z\cdot z  \leq C,$ 
\smallskip

(ii)~for any $x\in \R^d$, $(\tilde u^N_t(x))_{t\in [0,T]}$ and $(\tilde \mu^N_t(x))_{t\in [0,T]}$ are adapted to the filtration $(\mathcal F^N_t)_{t\in [0,T]}$ and,  therefore,  the filtration $(\mathcal F_t)_{t\in [0,T]}$, and $(\tilde M^N_{t_n}(x))_{n=0,\ldots, N}$ is a martingale with respect to  the discrete filtration $(\mathcal F_{t_n})_{n=0,\ldots, N}$.
\end{lem}


\begin{proof} We show  first $(\tilde u^N, \tilde M^N, \tilde \mu^N)$ is well-posed. For this, we define by backward induction a sequence of maps $\tilde U^N:\R^d\times [t^N_n, t^N_{n+1})\times \mathcal P_1(\R^d)\times \Omega \to \R$ such that, for each $t\in [t^N_n, t^N_{n+1})$,  $\tilde U^N_t$ is $\mathcal F_{t^N_n}-$measurable,  H\"{o}lder continuous in $m$ uniformly in $x$, bounded and semiconcave in  $x$ uniformly in $m$ and strongly monotone in the sense of \eqref{takis22}. 
\smallskip

We set $\tilde U^N_N= \tilde G$ and, given $\tilde U^N_{t^N_{n+1}}$, we  define $\tilde U^N_t$ on $[t^N_n, t^N_{n+1})$  as follows:  for any $\bar m_0\in \mathcal P_2\cap L^\infty$ we solve the MFG system
$$
\left\{\begin{array}{l}
- \partial_t \tilde v_t+\tilde H^N_t (D\tilde v_t, x)= \tilde F^N(x,\tilde m_t) \ \ {\rm in}\ \  \R^d\times (t^N_n,t^N_{n+1}), \\[2mm]
\partial_t \tilde m_t -{\rm div} (\tilde  m_t D_p\tilde H^N_t (D\tilde v_t, x))=0  \ \ {\rm in}\ \  \R^d\times (t^N_n,t^N_{n+1}),\\[2mm]
\tilde m_{t_n}= \bar m_0 \ \  \text{nd} \ \   \tilde v_{t^{N,-}_{n+1}}= \E\left[ \tilde U^N_{t^N_{n+1}}(\cdot,\tilde m^N_{t^{N,+}_{n+1}})\ |\ \mathcal F_{t^N_n} \right] \ \ \text{in} \ \ \R^d.
\end{array}\right. 
$$
We know from the discussion after Proposition~\ref{prop.MFGdet} that, if we set $\tilde U^N_t(x,\bar m_0)= \tilde v_{t}(x)$, then $\tilde U^N_t$ can be extended on $\R^d \times \mathcal P_1$ and satisfies the required regularity properties. 
\smallskip

In what follows it will be convenient to set $\tilde\rho^{N,n}_t(x,\bar m_0)= \tilde m_t(x)$. We remark  that $\tilde\rho^{N,n}$ is $\mathcal F_{t_n}-$measurable and continuous in $\bar m_0$ in $\mathcal P_1(\R^d)$. 
\smallskip

Given $\bar m_0\in \mathcal P_2\cap L^\infty$, we now build $(\tilde u^N, \tilde \mu^N, \tilde M^N)$. We set, for $t\in [0, t^N_1]$,  
$$
(\tilde u^N_t(x), \tilde \mu^N_t(x))= (\tilde U^N_t(x,\bar m_0), \tilde \rho^{N,0}_t(x,\bar m_0)),
$$
and note that, in view of  Lemma~\ref{lem.MFGdet}, $\tilde \mu^N_t$ is bounded in $L^\infty$ and has bounded finite second order moment both uniformly in $N$.  
Then, using a forward in time induction, we define, for $t\in [t^N_n,t^N_{n+1})$, 
$$
(\tilde u^N_t(x), \tilde \mu^N_t(x))= (\tilde U^N_t(x,\tilde \mu^N_{t^N_n}), \tilde \rho^{N,n}_t(x,\tilde \mu^{N}_{t^N_n})),
$$
and 
$$
\Delta\tilde M^N_t(x)  = \tilde u^N(x,t^{N,-}_{n+1})- \E\left[ \tilde u^N(x,t^{N,+}_{n+1}) | {\mathcal F}_{t^N_n}\right]\ \ {\rm on}\ \ (t^N_n, t^N_{n+1}),\\
$$
and 
$$
\tilde M^N_t(x) = \sum_{t^N_n<t} \Delta \tilde M^N_{t^N_n}(x). 
$$
In view of the  definition of $\tilde U^N$ and $\rho^{N,n}$, the triplet $(\tilde u^N, \tilde M^N, \tilde \mu^N)$ solves the required equations and is adapted to the discrete filtration $(\mathcal F^N_t)_{t\in [0,T]}$. The estimates on $\tilde u^N$ and $\tilde \mu^N$ follow from Lemma \ref{lem.MFGdet} applied on each time interval $(t^N_n, t^N_{n+1})$. We remark that  this is the place  we use that  the estimates in Lemma \ref{lem.MFGdet} is grow linearly in time. The bound on $\tilde M^N$ is  obtained as in the proof of Lemma \ref{lem.reguuNMN}. 
\end{proof}
\subsection{Passing to the limit}\label{subsec.lim}

The aim  is to pass to the limit in the discrete MFG system. Using the strong monotonicity of $\tilde F$ and $\tilde G$, we obtain the following estimate.

\begin{lem}\label{lem.hqens} Let $(\tilde u^N, \tilde M^N, \tilde \mu^N)$ be defined as above. There exists a random variable $\omega^N$ such that $\underset{N\to \infty} \E\Bigl[ \omega^N\Bigr]=0$ and, for $K, N\in \N$ with $K\geq N$,  
\begin{align*}
&  \E\Bigl[ \|\tilde G(\cdot,\tilde \mu^N_T)-\tilde G(\cdot,\tilde \mu^K_T)\|_\infty^{d+2} +\int_0^T\|\tilde F_t(\cdot, \tilde \mu^N_t)-\tilde F_t(\cdot, \tilde \mu^K_t)\|_\infty^{d+2} dt\Bigr]\leq C\E\Bigl[ \omega^N\Bigr].
\end{align*}
\end{lem}

\begin{proof} Using the fact that the pair $(\tilde u^N, \tilde \mu^N)$ is piecewise classical solution of the MFG system, with $\tilde \mu^N$ continuous in time and adapted, we find, following the classical  Lasry-Lions computation, that 
\begin{align*}
& \frac{d}{dt} \E\Bigl[ \int_{\R^d} (\tilde u^N_t(x)-\tilde u^K_t(x))(\tilde \mu^N_t(x)-\tilde \mu^K_t(x)) dx \Bigr] \\[2mm] 
& =  \E\Bigl[ \int_{\R^d} (\tilde H^N_t (D\tilde u^N_t(x),x)-\tilde F^N_t(x, \tilde \mu^N_t)-\tilde H^K_t (D\tilde u^K_t(x),x)+\tilde F^K_t(x, \tilde \mu^K_t)) (\tilde \mu^N_t(x)-\tilde \mu^K_t(x)) dx \\[2mm]
&  + \int_{\R^d} (\tilde u^N_t(x)-\tilde u^K_t(x))({\rm div}(\tilde \mu^N_t(x)D_p\tilde H^N_t(D\tilde u^N_t(x),x))- {\rm div}(\tilde \mu^K_t(x)D_p\tilde H^K_t(D\tilde u^K_t(x),x)))dx \Bigr]\\[2mm]
& =  -\E\Bigl[ \int_{\R^d} (\tilde F^N_t(x, \tilde \mu^N_t)-\tilde F^K_t(x, \tilde \mu^K_t)) (\tilde \mu^N_t(x)-\tilde \mu^K_t(x)) dx\Bigr]\\ 
&  - \E\Bigl[ \int_{\R^d} \mu^N_t(x)\Bigl(\tilde H^K_t (D\tilde u^K_t(x),x) -  \tilde H^N_t (D\tilde u^N_t(x),x) \\
& - (D\tilde u^N_t(x)- D \tilde u^K_t(x))\cdot D_p\tilde H^N_t(D\tilde u^N_t(x),x)\Bigr) dx  + \int_{\R^d} \mu^K_t(x)\Bigl( \tilde H^N_t (D\tilde u^N_t(x),x) \\[2mm]
&-\tilde H^K_t (D\tilde u^K_t(x),x) - (D\tilde u^K_t(x)- D \tilde u^N_t(x))\cdot D_p\tilde H^K_t(D\tilde u^K_t(x),x)\Bigr) dx \Bigr].
\end{align*}

In order to use the strong monotonicity assumption on $\tilde F$ and the convexity of $\tilde H$, we replace the discretized maps $\tilde F^N$ and $\tilde H^N$ by the continuous ones and  find 
\begin{align*}
& \frac{d}{dt} \E\Bigl[ \int_{\R^d} (u^N_t(x)-u^K_t(x))(\tilde \mu^N_t(x)-\tilde \mu^K_t(x)) dx \Bigr] \\[2mm] 
& \leq  -\E\Bigl[ \int_{\R^d} (\tilde F_t(x, \tilde \mu^N_t)-\tilde F_t(x, \tilde \mu^K_t)) (\tilde \mu^N_t(x)-\tilde \mu^K_t(x)) dx\Bigr]+C\E\Bigl[ \omega^N\Bigr] ,
\end{align*}
where, with $C$ being the uniform bound on $D\tilde u^N$ and on $D\tilde u^K$,
\be\label{defOmegaNBis}
\omega^N= \sup_{x\in \R^d, |t-s|\leq 1/N, |p|\leq C, m\in {\mathcal P}_1} \Big[|\tilde F_t(x,m)-\tilde F_s(x,m)| + |\tilde H_t(p,x)-\tilde H_s(p,x)|\Big]. 
\ee
\smallskip

Integrating in time inequality above and using the fact that $\mu^N_0=\mu^K_0$, we get
\begin{align*}
&  \E\Bigl[\int_{\R^d} (\tilde G(x,\tilde \mu^N_T)-\tilde G(x,\tilde \mu^K_T))(\tilde \mu^N_T(x)-\tilde \mu^K_T(x)) dx \Bigr] \\[2mm] 
& +\E\Bigl[\int_0^T \int_{\R^d} (\tilde F_t(x, \tilde \mu^N_t)-\tilde F_t(x, \tilde \mu^K_t)) (\tilde \mu^N_t(x)-\tilde \mu^K_t(x)) dxdt\Bigr]\leq C\E\Bigl[ \omega^N\Bigr] .
\end{align*}
Therefore, in view of the  strong monotonicity  of $\tilde F$ and $\tilde G$ in \eqref{takis22} and \eqref{takis23},  we find 
\begin{align*}
&  \E\Bigl[ \int_{\R^d} (\tilde G(x,\tilde \mu^N_T)-\tilde G(x,\tilde \mu^K_T))^2 dx +\int_0^T \int_{\R^d} (\tilde F_t(x, \tilde \mu^N_t)-\tilde F_t(x, \tilde \mu^K_t))^2 dxdt\Bigr]\leq C\E\Bigl[ \omega^N\Bigr] .
\end{align*}
Finally, the uniform Lipschitz regularity of $\tilde F$ and $\tilde G$ in space and an elementary interpolation yield  
\begin{align*}
&  \E\Bigl[ \sup_{x}|\tilde G(x,\tilde \mu^N_T)-\tilde G(x,\tilde \mu^K_T)|^{d+2} +\int_0^T\sup_{x}|\tilde F_t(x, \tilde \mu^N_t)-\tilde F_t(x, \tilde \mu^K_t)|^{d+2} dt\Bigr]\leq C\E\Bigl[ \omega^N\Bigr] .  
\end{align*}
\end{proof}

Next we estimate the difference between $\tilde u^N$ and $\tilde u^K$. 

\begin{lem}\label{lem.tildeuNCauchy} The sequence $(\tilde u^N)_{N\geq 1}$ is a Cauchy with the respect  to the family of seminorms 
$$
\left(\sup_{t\in [0,T]}  \E\left[\|\tilde u_t\|_{L^\infty (B_R)}^{d+1}\right]\right)_{R>0}
$$

\end{lem}

\begin{proof} Since the arguments are almost identical to the ones used in the proof of Theorem \ref{thm.mainexists} and Proposition \ref{prop.comparison}, here we only present a sketch.
\smallskip
 
Let  $1<N<K$ and $\phi:\R\to \R^+$ smooth, Lipschitz continuous, convex and nonincreasing map, and  set  $w_t(x)= \phi(\tilde u^N_t(x)-\tilde u^K_t(x))$. Using induction and the  convexity of $\phi$ to cancel the jump terms, which  are martingales,  
for any $t\in [0,T)$ and $h\in(0,T-h)$, we find 
\begin{align*}
& \E\Bigl[ w_{t+h}(x)- w_t(x) \Bigr] \\
& \geq  \E\Bigl[  \int_t^{t+h} \phi'(\tilde u^N_s(x)-\tilde u^K_s(x)) (\tilde H^N_s(D\tilde u^N(x)_s,x)-\tilde F^N_s(x,\tilde \mu^N_s)\\
& -\tilde H^K_s(D\tilde u^K_s(x),x)+\tilde F^N_s(x,\tilde \mu^K_s)) ds \Bigr]
 \geq \E\Bigl[   \int_t^{t+h} ({\bf b}_s(x)\cdot D w_s(x) + \zeta_s(x))ds\Bigr], 
\end{align*}
where 
\begin{align*}
& \zeta_s(x)=  \phi'(\tilde u^N_s(x)-\tilde u^K_s(x))(\tilde H^N_s(D\tilde u^K_s(x),x)-\tilde H^K_s(D\tilde u^K_s(x),x)-\tilde F^N_s(x,\tilde \mu^K_s)+\tilde F^N_s(x,\tilde \mu^N_s))\\[2mm]
& \text{and}\\
& {\bf b}_s(x)= \int_0^1 D_p\tilde H^N_s((1-\lambda)D\tilde u^N_s(x)+\lambda D\tilde u^K_s(x))d\lambda. 
\end{align*}  
For some $\alpha, \beta>0$ to be chosen below, let 
$$
e_t = \E\left[ \int_{B_{\alpha+\beta t}} w_t(y)dy\right]. 
$$
As in the proof of Theorem \ref{thm.mainexists} and for $\beta$  large enough, but independent of $N$ and $K$, we get 
$$
e_t\leq C (e_T+ \E\Bigl[\int_t^T \int_{B_{\alpha+\beta s}} \zeta_s(y)dyds\Bigr]).
$$
Choosing, after approximation,  $\phi(s)= (-s)_+$, we derive that, for all $t\in [0,T]$,  $\omega^N$ is as in \eqref{defOmegaNBis}, and constants $C>0$ and $C_{\alpha, \beta}=C(\alpha, \beta)>0$, 
\begin{align*}
&\E\left[ \int_{B_{\alpha+\beta t}} (- (\tilde u^N_t(y)- \tilde u^K_t(y)))_+dy\right] \leq C \E\left[ \int_{B_{\alpha+\beta t}} (- (\tilde G(y,\tilde \mu^N_T)-  \tilde G(y, \tilde \mu^K_T)))_+dy\right]\\
& \qquad +  
C \E\Bigl[ \int_t^T\int_{B_{\alpha+\beta s}} |\tilde F_s(y,\tilde \mu^N_T)-  \tilde F_s(y, \tilde \mu^K_T)| dyds\Bigr]
+ C_{\alpha,\beta} \E[\omega^N]. 
\end{align*}
Reversing the roles of $u^N$ and $u^K$, we then obtain, for all $t\in [0,T]$, 
\begin{align*}
&\E\left[ \int_{B_{\alpha+\beta t}} |\tilde u^N_t(y)- \tilde u^K_t(y)|dy\right] \leq C \E\left[ \int_{B_{\alpha+\beta t}} |\tilde G(y,\tilde \mu^N_T)-  \tilde G(y, \tilde \mu^K_T)|dy\right]\\
& \qquad +  
C \E\Bigl[ \int_t^T\int_{B_{\alpha+\beta s}} |\tilde F_s(y,\tilde \mu^N_T)-  \tilde F_s(y, \tilde \mu^K_T)| dyds\Bigr]
+ C_{\alpha,\beta} \E[\omega^N]. 
\end{align*}
It follows from Lemma \ref{lem.hqens}, for some $\ep_{\alpha,\beta}(N)\to 0$ as $N\to+\infty$, 
\begin{align*}
&\E\left[ \int_{B_{\alpha+\beta t}} |\tilde u^N_t(y)- \tilde u^K_t(y)|dy\right] \leq \ep_{\alpha,\beta}(N). 
\end{align*}
The uniform Lipschitz estimate for the $\tilde u^N$ gives the result. 
\end{proof}

We have now established all the ingredients needed for the proof of the existence and uniqueness of solution of the stochastic MFG system.
 \begin{proof}[Proof of Theorem \ref{thm.main}] 
Lemma~\ref{lem.tildeuNCauchy} and the properties of $\tilde F$ and $\tilde G$ yield that the sequences $(\tilde u^N)_{N\in \N}$, $(\tilde F(\cdot, \tilde \mu^N_\cdot))_{N\in \N}$ and $\tilde  G(\cdot,\tilde \mu^N_T)_{N\in \N}$ are   Cauchy, with respective limits  $\tilde u$, $\tilde f$ and $\tilde g$.
\smallskip

It follows, as in the proof of Theorem \ref{thm.mainexists}, that the sequences $(D\tilde u^N)_{N\in \N}$ and $(\tilde M^N)_{N\in \N}$ also converge, as $N\to \infty$, to $D\tilde u$ and $\tilde M$ respectively, $\tilde M$ is a continuous process, and,  in addition, $(\tilde u,\tilde M)$ solves 
$$
d\tilde u_t = (\tilde H_t(D\tilde u_t,x)-\tilde f_t(x))dt+ d\tilde M_t \ \text{in}  \ \R^d\times[0,T)  \ \ \  \tilde u_T= \tilde g \ \text{on} \ \R^d. 
$$
Next we need  to  check that the sequence $(\tilde \mu^N)_{N\in \N}$ has a limit $\tilde \mu$ and that   $\tilde f_t(x)=\tilde F_t(x,\tilde \mu_t)$ and $\tilde g(x)= \tilde G(x,\tilde \mu_T)$. 
\smallskip

Fix $\omega$ for which $\tilde u^N$ converges to $\tilde u$ locally uniformly and $D\tilde u^N$ converges to $D\tilde u$ a.e.. In view of the bound on $(\tilde \mu^N)_{N\in \N}$,  the sequence $(\tilde \mu^N)_{N\in \N}$ is relatively compact in  $C([0,T]; {\mathcal P}_1(\R^d))$ and in $L^\infty$ weak-*. So 
we can find a subsequence, which we  denote in the same way, which converges, in  $C([0,T]; {\mathcal P}_1(\R^d))$ and in $L^\infty-$weak-*, to some $\tilde \mu$, which is a bounded solution 
of the continuity equation 
\[ \partial_t\tilde \mu_t =  {\rm div}(\tilde \mu_tD_p\tilde H_t(D\tilde u_t,x))\  \text{in}  \ \R^d\times(0,TT  \ \ \  \tilde \mu_0=\bar m_0 \ \text{on} \R^d.\]
 In addition, the continuity of $\tilde F_t$ with respect to the measure argument implies that  $(\tilde F_t(x, \tilde \mu^N_t))$ converges to $\tilde F_t(x, \tilde \mu_t)$ for any $(x,t)\in [0,T]\times \R^d$. Hence,   
 \be\label{takis3333}
 \tilde F_t(x, \tilde \mu_t)= \tilde f_t(x).
 \ee
Since \eqref{takis3333} holds true along any converging subsequence, we infer from   the strict monotonicity of  $\tilde F_t$
that the compact sequence $(\tilde \mu^N)_{N\in \N}$ has a unique accumulation point in $C([0,T]; {\mathcal P}_1(\R^d))$, and, thus it converges a.s. to an adapted and bounded process $\tilde \mu$ satisfying the continuity equation. 
 \smallskip

It follows that  the sequences $(\tilde F^N_t(x, \tilde \mu^N_t))_{N\in \N}$ and $(\tilde G^N(x, \tilde \mu^N_T))_{N\in \N}$ converge locally uniformly to $\tilde F_t(x, \tilde \mu_t)$ and $\tilde G(x, \tilde \mu_T)$ respectively. We can therefore conclude that the pair $(\tilde u,\tilde M, \tilde \mu)$ is a solution to the MFG system \eqref{stoMFG_Intro}. 
\smallskip

Since the proof of the uniqueness of solutions follows the standard argument,  we only sketch it. Let ($\tilde u^1,\tilde M^1, \tilde m^1)$ and $(\tilde u^2,\tilde M^2,\tilde m^2)$ be two solutions of \eqref{stoMFG_Intro}. We show that 
\begin{align}
&\E\Bigl[ \int_0^T \int_{\R^d} (\tilde F_t(x,\tilde m^1_t) - \tilde F_t(x,\tilde m^2_t))(\tilde m^1_t-\tilde m^2_t)(dx)dt  \label{jelrlznesrdd}\\
& \qquad \qquad  \qquad  \qquad + \int_{R^d} (\tilde G(x,\tilde m^1_T) - \tilde G(x,\tilde m^2_T))(\tilde m^1_T-\tilde m^2_T)(dx)\Bigr] \leq 0. \notag
\end{align}
For this, for $n=0,\ldots,N$,   let $t^N_n=nT/N$, and  note that, in view of   the equation satisfied by  $\tilde u = \tilde u^1-\tilde u^2$, we have, letting $\tilde M= \tilde M^1-\tilde M^2$,
\begin{align*}
 \tilde u_{t^N_{n+1}}(x)  -\tilde u_{t^N_{n}}(x)&= -\int_{t^N_n}^{t^N_{n+1}}( \tilde H_t(D\tilde u^1,x)-\tilde F_t(x,\tilde m^1(t))- \tilde H_t(D\tilde u^2,x)+\tilde F_t(x,\tilde m^2(t))dt\\
 & \qquad - (\tilde M_{t^N_{n+1}}(x)-\tilde M_{t^N_n}(x)).
\end{align*}

Let $\tilde m=\tilde m^1-\tilde m^2$. Integrating the equality above against $\tilde m_{t^N_n}$ and summing over $n$ gives
\begin{align*}
& \sum_{n=0}^{N-1}  \int_{\R^d} (\tilde u_{t^N_{n+1}}(x)- \tilde u_{t^N_{n}}(x))\tilde m_{t^N_n}(x)dx \\
& \qquad = - \sum_{n=0}^{N-1}  \int_{t^N_n}^{t^N_{n+1}}\int_{\R^d}( \tilde H_t(D\tilde u^1(x),x)-\tilde F_t(x,\tilde m^1_t)- \tilde H_t(D\tilde u^2_t(x),x)+\tilde F_t(x,\tilde m^2_t))\tilde m_{t^N_n}(x)dxdt  \\
& \qquad\qquad\qquad -\sum_{n=0}^{N-1}  \int_{\R^d} (\tilde M_{t^N_{n+1}}(x)-\tilde M_{t^N_n}(x)) \tilde m_{t^N_n}(x)dx . 
\end{align*}
After reorganizing  the left-hand side above by taking into account the equation satisfied by $\tilde m$ yields 
\begin{align*}
& \sum_{n=0}^{N-1}  \int_{\R^d} (\tilde u_{t^N_{n+1}}(x)- \tilde u_{t^N_{n}}(x))\tilde m_{t^N_n}(x)dx \\
&  = \int_{\R^d} \tilde u_{T}(x)\tilde m_{t^N_{N-1}}(x)dx-  \int_{\R^d} \tilde u_0(x)\tilde m_0(x)dx-  \sum_{n=1}^{N-1}  \int_{\R^d} \tilde u_{t^N_{n}}(x)(\tilde m_{t^N_n}(x)-\tilde m_{t^N_{n-1}}(x))dx \\
&  = \int_{\R^d} (\tilde G(x, \tilde m^1_T)-\tilde G(x,\tilde m^2_T))(\tilde m^1_{t^N_{N-1}}(x)-\tilde m^2_{t^N_{N-1}}(x))dx \\
& \qquad + \sum_{n=1}^{N-1}\int_{t^N_{n-1}}^{t^N_n}  \int_{\R^d} D\tilde u_{t^N_{n}}(x)\cdot (D_p \tilde H_{t}(D\tilde u^1_t(x),x)\tilde m^1_t(x)-D_p \tilde H_{t}(D\tilde u^2_t(x),x)\tilde m^2_t(x))dxdt. 
\end{align*}
We let $N\to+\infty$ and find, after taking expectation,  
\begin{align*}
& \E\Bigl[ \int_{\R^d} (\tilde G(x, \tilde m^1_t)-\tilde G(x,\tilde m^2_T))(\tilde m^1_T(x)-\tilde m^2_T(x))dx \\
& \qquad +\int_0^T \int_{\R^d} D\tilde u_{t}(x)\cdot (D_p \tilde H_{t}(D\tilde u^1_t(x),x)\tilde m^1_t(x)-D_p \tilde H_{t}(D\tilde u^2_t(x),x)\tilde m^2_t(x))dxdt\Bigr]\\
& =  -\E\Bigl[\int_{0}^{T}\int_{\R^d}( \tilde H_t(D\tilde u^1_t(x),x)-\tilde F_t(x,\tilde m^1_t)- \tilde H_t(D\tilde u^2_t(x),x)+\tilde F_t(x,\tilde m^2_t))\tilde m_{t}(x)dxdt  \Bigr] .
\end{align*}
We can now  rearrange the expression in the usual way, and taking into account the convexity of $\tilde H_t=\tilde H_t(p,x)$ in $p$, to conclude that \eqref{jelrlznesrdd} holds. 
\smallskip

Using  the strict monotonicity assumption on $\tilde F$ we infer that, $\P-$a.s. and for a.e. $(x,t)\in [0,T]\times \R^d$,  $\tilde m^1_t=\tilde m^2_t$. Thus  $\tilde u^1$ and $\tilde u^2$ solve the same HJ equation. It follows from Proposition \eqref{prop.comparison} that $\tilde u^1=\tilde u^2$.  
\smallskip

The equality $\tilde M^1=\tilde M^2$ follows from the equation satisfied by the $\tilde u^i$. 
\end{proof}

\subsection{Application to $N-$player differential games}\label{subsec.game}

We consider here  a game with $N$ players and show that, if $N$ is large enough, the optimal controls given by the solution of the stochastic MFG system \eqref{stoMFG_Intro}  provide an approximate Nash equilibrium for the game. 
\smallskip

We begin with the notation, terminology and the general setting.  In what follows,  $N\in \N$, $\bar m_0\in \mathcal P_2(\R^d)$ with an $L^\infty-$density,  $(Z^{i})_{i\in \N}$ is a sequence of independent random initial conditions on $\R^d$  with law $\bar m_0$, and $W$ is a Brownian motion independent of the $(Z^{i})_{i\in \N}$.  
\smallskip

The state $X^{\alpha^i}$  of the $i-$th player  satisfies, for $i=1,\ldots,N$,  the stochastic differential equation 
\be\label{eq.XNi}
dX^{\alpha^i}_t= \alpha^{i}_tdt + \sqrt{2\beta}dW_t \ \ \text{in} \ \  [0,T] \ \ \ \  X^{\alpha^i}_0=Z^{i},
\ee
with  $\alpha^{i}$ an admissible control of player $i$, that is, an $\R^d-$valued  measurable process  adapted to the filtration generated by $(W_s)_{s\leq t}$ and the $(Z^j)_{j\in \N}$, and  such that $\E[\int_0^T |\alpha^{i}_t|^2dt]<+\infty$. 
Note that the noise $W$ is the same for all the players. 
\smallskip

The cost of player $i$, associated to the admissible control $\alpha^i$ and given the admissible controls of the $(\alpha^{j})_{j\neq i}$ of the other players, is 
$$
J^{N,i}(\alpha^{i}, (\alpha^{j}_{j\neq i}))= \E\left[ \int_0^T (L(\alpha^{i}_t, X^{\alpha^i}_t)+ F(X^{\alpha^i}_t, m^{N,i}_{{\bf X}_t}))dt + g(X^{\alpha^i}_T, m^{N,i}_{{\bf X}^N_T})\right], 
$$
where ${\bf X}_t =(X^{\alpha^1}_t, \dots, X^{\alpha^N}_t)$ with $X^{\alpha^j}$ a solution of \eqref{eq.XNi} and 
$$
m^{N,i}_{{\bf X}_t}= \frac{1}{N-1}\sum_{j=1,j\neq i}^N \delta_{X^{\alpha^j}_t}.
$$
Given $\ep>0$, we say that the family  $(\bar \alpha^i)_{i=1,\ldots, N}$ of admissible controls is an $\epsilon-$ Nash equilibrium of the game, if, for any $i=1, \dots, N$ and for any admissible control $\alpha^i$ of the player $i$, 
$$
J^{N,i}(\bar \alpha^i, (\bar \alpha^{j})_{j\neq i}) \leq J^{N,i}( \alpha^i, (\bar \alpha^{j})_{j\neq i})+\ep. 
$$ 
The  associated Hamiltonian of is 
$$
H(p,x)= \sup_{\alpha\in \R^d}[ -p\cdot \alpha- L(\alpha, x)].
$$
We assume that  
\be\label{takis30}
\begin{cases}
F,G:\R^d\times \mathcal P_1\to \R \ \text{are globally Lipschitz continuous and, for some $C_0>0$},\\[2mm] 
\qquad \qquad \underset{m\in \mathcal P_1} \sup  \|F(\cdot, m)\|_{C^2}+\|G(\cdot, m) \|_{C^2} \leq C_0,
\end{cases}
\ee
\be\label{takis31}
\text{$F$ and $G$ are strongly monotone with constant $\alpha$ and $F$ is strictly monotone}
\ee
and the Hamiltonian $H$ is of the form
\be\label{takis32}
\begin{cases}
H(p, x)= \frac{a(x)}{2}|p|^2 + B(x)\cdot p \\[2mm]
\text{with $a\in C^2(\R^d)$, $B\in C^2(\R^d;\R^d) $ and  $C_0^{-1}\leq a(x)\leq C_0$.}
\end{cases}
\ee  
It is then immediate that  $$L(\alpha, x)= \dfrac{1}{2a(x)}|\alpha+B(x)|^2.$$

As in the earlier parts  of the paper, we set
\be\label{takis33} 
\begin{cases}
\tilde H_t(x,p)= H(p, x+\sqrt{2\beta} W_t),\\[2mm]
\tilde F_t(x,m)= F(x+\sqrt{2\beta}W_t,(id+\sqrt{2\beta}W_t)\sharp m_t)  \ \ \text{and}\\[2mm]
\tilde G(x)= G(x-\sqrt{2\beta} W_T, (id+\sqrt{2\beta}W_T)\sharp m_T).
\end{cases}
\ee

 In view of the conditions above, $\tilde H$, $\tilde F$, $\tilde G$ and $\bar m_0$ satisfy \eqref{H'1},  \eqref{H'2}, \eqref{H'3}, \eqref{H'4}, \eqref{H'5}, and \eqref{H'6}.
\smallskip

We denote by $(\tilde u, \tilde M, \tilde m)$ the solution  of  \eqref{stoMFG_Intro}, and recall that 
Proposition \ref{prop.ExistOptTraj} implies, for a.e. $x\in \R^d$, the  existence of  a family $(\bar \gamma^{x}_t)_{t\in [0,T]}$ of adapted processes which minimize 
\be\label{kejhsnrdt=fv}
\tilde u_0(x) =\inf_{\gamma_0=x}
 \E\left[ \int_t^T (\tilde L_s(\dot \gamma_s,\gamma_s)+ \tilde F_s(\gamma_s, \tilde m_s))ds +\tilde G(\gamma_T, \tilde m_T)\right],
\ee
where 
$$\tilde L_t(p, \alpha,x, \omega)= \dfrac{1}{2\tilde a_t(x)}|p +\tilde B_t(x)|^2 \ \  \text{and} \ \ \tilde a_t(x)= a(x+\sqrt{2\beta}W_t), \  \tilde B_t(x)= B(x+\sqrt{2\beta}W_t).$$
Set $\bar \alpha^x_t = \dot{\bar \gamma}^x_t= -D_p\tilde H_t(D\tilde u_t(\bar \gamma^x_t), \bar \gamma^x_t)$. 

\begin{prop}\label{prop.cv} Assume  \eqref{H'1}, \eqref{takis30}, \eqref{takis31} and \eqref{takis32}.
Then,  for any $\ep>0$, there exists $N_0=N_0(\ep)\in \N$ such that, for any $N\geq N_0$, the family of random controls $(\bar \alpha^{Z^{i}})_{i=1, \dots, N}$ is an $\epsilon-$Nash equilibrium of the game.
\end{prop}

\begin{proof} To simplify the notation, in what follows we  set $\bar X^i= X^{\bar \alpha^{Z^i}}$ and note that 
$$
\bar X^i_t= Z^i+ \int_0^t \bar \alpha^{Z^i}_sds + \sqrt{2\beta}W_t = \bar \gamma^{Z^i}_t + \sqrt{2\beta}W_t. 
$$
We check that the conditional law of $(\bar X^i_t)_{t\in [0,T]}$ given $(W_t)_{t\in [0,T]}$ is $m_t= (Id+\sqrt{2\beta}W_t)\sharp \tilde m_t$. 
Indeed, in view of  Proposition II-2.7 of \cite{CaDeBook} and since $\bar \gamma^{Z^i}$ solves the ODE
$$
\dot{\bar \gamma}^{Z^i}_t=  -D_p\tilde H_t(D\tilde u_t(\bar \gamma^{Z^i}_t), \bar \gamma^{Z^i}_t),
$$
the conditional law $\tilde \mu_t$ of $(\bar \gamma^{Z^i}_t)$ given $W$ is a solution in the sense of distributions of  the continuity equation 
$$
\partial_t \tilde \mu_t -{\rm div}(\tilde \mu_t D_p\tilde H_t(D\tilde u_t(x),x))=0 \ \ \text{in} \ \ \R^d\times (0,T)  \ \ \  \ \tilde \mu_0=\bar m_0. 
$$
It follows from  Proposition \ref{prop.UnikCont} that this solution is unique.
Therefore,  the conditional law $\tilde \mu_t$ of $(\bar \gamma^{Z^i}_t)_{t\in [0,T]}$ given $W$ is $\tilde m_t$. Since $\bar X^i_t= \bar \gamma^{Z^i}_t + \sqrt{2\beta}W_t$, this implies that the conditional law of  $\bar X^i_t$ given $W$ is $m_t$. 
\smallskip

Since the $Z^i$'s   and $W$ are independent, it  is clear that the $\bar X^i$'s  are conditionally independent and have  the same law $m$ given $W$.  It then follows from the  Glivenko-Cantelli law of large numbers that, $\P-$a.s., 
$$
\lim_{N\to+\infty} \E\left[{\bf d}_2(m^{N,i}_{\bar{\bf X}_t}, m_t)\ |\ W \right] =0. 
$$
In view of the Lipschitz continuity and boundedness of $F$ and $G$, the limit above  implies that, for any $(x,t)$, 
$$
\lim_{N\to+\infty} \E\left[|F(x, m^{N,i}_{\bar{\bf X}_t})- F(x, m_t)|+|G(x,m^{N,i}_{\bar{\bf X}_T})- G(x, m_T)|  \right] =0.
$$
As the integrant is uniformly continuous in $x$ uniformly in $t$ and $N$ and has a uniform modulus in $t$, which in expectation is uniform in $x$ and $N$, we deduce  that, for any $R>0$ and $\ep>0$, there exists $N_R\in \N$ such that, if $N\geq N_R$,
\be\label{elzkrsjd}
\E\left[\sup_{x\in B_R, t\in [0,T]}|F(x, m^{N,i}_{\bar{\bf X}_t})- F(x, m_t)|+|G(x,m^{N,i}_{\bar{\bf X}_T})- G(x, m_T)|  \right] \leq \frac{\ep}{4(T+1)}.
\ee

Note that, since $D\tilde u$ is bounded, there exists   $M>0$ such  that $\bar \alpha^{Z^i}_t= -D_p\tilde H_t(D\tilde u_t(\bar \gamma^{Z^i}_t,t),\bar \gamma^{Z^i}_t)$ is bounded by $M$ and, thus,  we have 
\be\label{kuzajesrd}
\begin{cases}
J^{N,i}(\bar \alpha^{Z^i}, (\bar \alpha^{Z^j})_{j\neq i}) = \E\Big[ \int_0^T \frac{1}{2a_t(\bar X^i_t)}|\bar \alpha^{Z^i}_t+B(\bar X^i_t)|^2+ 
 F(\bar X^i_t, m^{N,i}_{\bar{\bf X}_t}) dt\\[2mm]
\hskip.9in  + G(\bar X^i_T, m^{N,i}_{\bar{\bf X}_T}) \Big] \leq T(C_0M^2+ C_0\|B\|_\infty^2+ \|F\|_\infty)+\|G\|_\infty.
\end{cases}
\ee
Let $\alpha^i$ be an admissible control for player $i$. To prove the claim, it is necessary to estimate $J^{N,i}(\alpha^i, (\bar \alpha^{Z^j})_{j\neq i})$ in terms  
$\E[\int_0^T |\alpha^i_s|^2ds]$.  
\smallskip

In what follows,  we introduce the constant 
$$
A=\max\Big[4TC_0(C_0M^2+2C_0\|B\|_\infty^2+2\|F\|_\infty)+8C_0\|G\|_\infty\; ,\;  \E[\int_0^T |\bar \alpha^{Z^i}_s|^2ds]\Big],
$$  
which is independent of $N$ and $i$, since the law of $\bar \alpha^{Z^i}_s$ does not depend on $i$. 
\smallskip

If $\E[\int_0^T |\alpha^i_s|^2ds]\geq A$, then, in view of \eqref{kuzajesrd} and the choice of $A$, we find 
\begin{align*}
J^{N,i}( \alpha^i, (\bar \alpha^{Z^j})_{j\neq i}) & = \E\left[ \int_0^T \frac{1}{2a_t(X^{\alpha^i}_t)}|\alpha^i_t+B( X^{\alpha^i}_t)|^2+ 
 F( X^{\alpha^i}_t, m^{N,i}_{\bar{\bf X}_t}) dt+ G( X^{\alpha^i}_T, m^{N,i}_{\bar{\bf X}_T}) \right] \\ 
 & \geq \frac{1}{4C_0}  \E\left[ \int_0^T |\alpha^i_t|^2dt \right]- TC_0 \|B\|_\infty^2- T\|F\|_\infty-\|G\|_\infty \\ 
 & \geq A/(4C_0) - TC_0 \|B\|_\infty^2- T\|F\|_\infty-\|G\|_\infty  \geq J^{N,i}(\bar \alpha^{Z^i}, (\bar \alpha^{Z^j})_{j\neq i}).
\end{align*}

If $\E[\int_0^T |\alpha^i_s|^2ds]\leq A$, an estimate which is satisfied by  $\bar \alpha^{Z^i}$, then,  for any $R>0$, we obtain 
\be\label{elzkrsjd2}
\begin{cases}
 \P\left[\underset{t\in [0,T]}\sup \; |X^{\alpha^i}_t| \geq R\right]  \\[2mm]
\leq  \P\left[|Z^i|\geq R/3 \right]
 + \P\left[\int_0^T |\alpha^i_t| dt\geq R/3 \right] + \P\left[\sqrt{2\beta} \underset{t\in [0,T]} \sup \; |W_t| \geq R/3 \right] \\[2mm]
 \leq 9R^{-2}\left(\E[|Z^i|^2]+ T\E\left[\int_0^T |\alpha^i_t|^2dt\right] +2\beta \E\left[ \underset{t\in [0,T]} \sup \; |W_t|^2\right]\right) \leq CR^{-2}(1+A).
\end{cases}
\ee
\smallskip

We fix $R$ large enough to be chosen below and $N\geq N_R$ as in \eqref{elzkrsjd}. Then 
\begin{align*}
& J^{N,i}( \alpha^i, (\bar \alpha^{Z^j})_{j\neq i})  = \E\left[ \int_0^T \frac{1}{2a(X^{\alpha^i}_t)}|\alpha^i_t+B( X^{\alpha^i}_t)|^2+ 
 F( X^{\alpha^i}_t, m^{N,i}_{\bar{\bf X}_t}) dt+ G( X^{\alpha^i}_T, m^{N,i}_{\bar{\bf X}_T}) \right] \\ 
 & \qquad \geq \E\left[ \int_0^T  \frac{1}{2a(X^{\alpha^i}_t)}|\alpha^i_t+B( X^{\alpha^i}_t)|^2+ 
 F( X^{\alpha^i}_t, m_t) dt+ G( X^{\alpha^i}_T, m_T) \right]  \\ 
 & \qquad\qquad - \E\left[ T\sup_t|F( X^{\alpha^i}_t, m^{N,i}_{\bar{\bf X}_t})-F( X^{\alpha^i}_t, m_t)| +|G( X^{\alpha^i}_T, m^{N,i}_{\bar{\bf X}_T})-G( X^{\alpha^i}_T, m_T)|\right],
 \end{align*}
where, in view of  \eqref{elzkrsjd} and \eqref{elzkrsjd2} and $R$ suffciently large,  
\begin{align*}
&\E\left[ T\underset{t\in [0,T]}\sup |F( X^{\alpha^i}_t, m^{N,i}_{\bar{\bf X}_t})-F( X^{\alpha^i}_t, m_t)| +|G( X^{\alpha^i}_T, m^{N,i}_{\bar{\bf X}_T})-G( X^{\alpha^i}_T, m_T)|\right]\\
& \qquad \leq (T+1) \E\left[ \sup_{x\in B_R, t\in [0,T]}|F(x, m^{N,i}_{\bar{\bf X}_t})- F(x, m_t)|+|G(x,m^{N,i}_{\bar{\bf X}_T})- G(x, m_T)|\right]
  \\
& \qquad\qquad +2(T+1)(\|F\|_\infty+\|G\|_\infty) \P\Big[ \sup_{t\in [0,T]}  |X^\alpha_t|\geq R\Big] \leq \dfrac{\ep}{4} +  \dfrac{C}{R^2}(1+A)<\dfrac{\ep}{2}.
 \end{align*}

 \smallskip
 
It follows that 
$$
 J^{N,i}( \alpha^i, (\bar \alpha^{Z^j})_{j\neq i})   \geq \E\left[ \int_0^T  \frac{1}{2a(X^{\alpha^i}_t)}|\alpha^i_t+B( X^{\alpha^i}_t)|^2+ 
 F( X^{\alpha^i}_t, m_t) dt+ G( X^{\alpha^i}_T, m_T) \right] -\dfrac{\ep}{2}.
$$
Therefore, setting $\gamma^{\alpha^i}_t:= X^{\alpha^i}_t-\sqrt{2\beta}W_t$, we get 
$$
 J^{N,i}( \alpha^i, (\bar \alpha^{Z^j})_{j\neq i})   \geq \E\left[\int_0^T  \tilde L(\alpha^i_t, \gamma^{\alpha^i}_t)+ 
\tilde F( \gamma^{\alpha^i}_t, \tilde m_t) dt+ \tilde G( \gamma^{\alpha^i}_T, \tilde m_T) \right] -\dfrac{\ep}{2}.
$$

 The same argument, with an estimate from above instead of an estimate from below, shows that 
$$
J^{N,i}(\bar  \alpha^{Z^i}, (\bar \alpha^{Z^j})_{j\neq i}) \leq  \E\left[ \int_0^T \tilde L(\bar \alpha^{Z^i}_t, \bar \gamma^{Z^i}_t)+ 
\tilde F( \bar \gamma^{Z^i}_t, \tilde m_t) dt+ \tilde G( \bar \gamma^{Z^i}_T,\tilde m_T) \right] +\dfrac{\ep}{2} .
$$
 Since, in view of   the optimality of $\bar \alpha^{Z^i}$ in \eqref{kejhsnrdt=fv}, we also have
 \begin{align*}
&    \E\left[\int_0^T  \tilde L(\alpha^i_t, \gamma^{\alpha^i}_t)+ 
\tilde F( \gamma^{\alpha^i}_t, \tilde m_t) dt+ \tilde G( \gamma^{\alpha^i}_T, \tilde m_T) \right] \\
&\qquad  \geq 
  \E\left[ \int_0^T \tilde L(\bar \alpha^{Z^i}_t, \bar \gamma^{Z^i}_t)+ 
\tilde F( \bar \gamma^{Z^i}_t, \tilde m_t) dt+ \tilde G( \bar \gamma^{Z^i}_T,\tilde m_T) \right],
 \end{align*}
combining  the three last inequalities  we conclude that $(\bar \alpha^{Z^{i}})_{i=1, \dots, N}$ is an $\epsilon-$Nash equilibrium.  
\end{proof}


\appendix 
\section{A uniqueness result for a continuity equation}

We study here the uniqueness of distributional solutions of  the forward continuity equation
\be\label{eq.CEappen}
\partial_t m +{\rm div}(mb)= 0 \ \ \text{in} \ \ \R^d\times (0,T) \ \ \  m(0)=\bar m_0 \ \ \text{in} \ \ \R^d,
\ee
where $\bar m_0$ is a Borel probability measure on $\R^d$ with a bounded density and $b:[0,T]\times \R^d\to \R^d$ is a bounded, half-Lipschitz from below Borel vector field, that is, there is a constant $C_0$ such that, for all $x,y \in \R^d$ and $t\in [0,T]$,
\be\label{takis40}
|b_t(x)|\leq C_0 \ \  \text{and} \ \  
(b(t,x)-b(t,y))\cdot(x-y) \geq -C_0|x-y|^2.
\ee
The existence and uniqueness of distributional solution of \eqref{eq.CEappen} is closely related to the existence and uniqueness of solutions of the ODE
\be\label{takis41}
\dot x_t= b_t(x_t)  \ \ \text{in} \ \ [0,t_0], \ \ \ x(t_0)=x_0.
\ee
The following result is a variant of a theorem in \cite{BeLaLi20}: 
\begin{prop}\label{prop.appen} Assume \eqref{takis40} and let $\bar m_0\in \mathcal P_1(\R^d)$ be absolutely continuous with  bounded density. Then there exists a unique bounded and absolutely continuous with respect to the Lebesgue measure solution of \eqref{eq.CEappen}. 
\end{prop}

The proof requires several steps and is based on the Filippov regularization $b^F$ of $b$; see Filippov \cite{Fi}. Recall that $b^F$ is an upper semicontinuous set-valued map with convex compact values. It follows from \eqref{takis40}, that,
if $\mathcal N=\{(x,t):  \{b_t(x)\}\neq b^F_t(x)\}$, then
\be\label{takis42}
\mathcal L^{d+1}(\mathcal N) =0. 
\ee
 

It is known that, in view of \eqref{takis40},  for any $(x_0,t_0)\in \R^d\times (0,T]$, there exists a unique absolutely continuous solution $X(x_0,t_0,\cdot)$ of the backward differential inclusion 
$$
\dot x_t\in b^F_t(x_t) \ \ \text{in} \ \ [0,t_0], \ \ \ x_{t_0}=x_0,
$$
which is referred to as the  Filippov solution of  the ODE, and, for all  $x_0, x_1\in \R^d$ and $t\in [0,t_0]$, 
\be\label{takis200}
|X(x_0,t_0,t)-X(x_1,t_0,t)| \leq  e^{C_0T} |x_0-x_1|.  
\ee

Note that any smooth approximation $b^\ep$ of $b$ obtained by, for example, a convolution with a nonnegative, smooth kernel, satisfies \eqref{takis40} with the same constant $C_0$. 
\smallskip

It then follows that the classical backward flow $X^\ep(x_0,t_0,\cdot)$  of 
$$
\dot x^\ep_t =b^\ep_t (x^\ep_t)  \ \ \text{in} \ \ [0,t_0], \ \ x^\ep_{t_0}=x,
$$
satisfies \eqref{takis200} with a uniform constant,  hence,  it converges locally uniformly to $X$.

%

\begin{lem}\label{lem.inject} Assume \eqref{takis40}. Then, for any $0\leq t\leq t_0\leq T$,  the map $x\to X(x, t_0,t)$ is surjective. 
\end{lem}

\begin{proof} Arguing  by contradiction, we assume that,  for some $t\in [0,t_0]$, there exists $y\in \R^d\backslash X(\R^d,t_0,t)$. The finite speed of propagation property and the continuity  the flow yield some $\delta>0$ such that $B_\delta(y)\cap X(\R^d,t_0,t)=\emptyset$.
\smallskip

Let $\bar m_0$ be a probability measure with a smooth density supported in $B_\delta(y)$ and $b^\ep$  a smooth approximation of $b$ satisfying \eqref{takis40} with a constant independent of $\ep$. Note that, for any $0\leq t\leq t_0\leq T$, the map $x\to X^\ep(x,t_0,t)$ is smooth and one-to-one.  
\smallskip

We consider the classical solution $m^\ep$ to the continuity equation 
$$
\partial_t m^\ep +{\rm div}(m^\ep b^\ep)=0 \ \ \text{in}  \ \ \R^d\times  (t,T)  \ \ \  m^\ep(t)=\bar m_0.
$$
Since ${\rm div}(b^\ep)\geq -C_0$, we infer from  the maximum principle that $$\|m^\ep\|_\infty\leq \|\bar m_0\|_\infty e^{C_0 T}.$$ 

Passing  (up to a subsequence) to the $\ep\to0$ limit, we obtain  $m\in L^\infty(\R^d\times (0,T))$  solving 
\be\label{eq.pierre}
\partial_t m +{\rm div}(mb)= 0 \ \ \text{in} \ \ \R^d\times (t,T) \ \ \  m(t)=\bar m_0 \ \ \text{in} \ \ \R^d,
\ee
in the sense of distributions.  
\smallskip

Next, we use Ambrosio's superposition principle \cite{Am04} which provides a connection between solutions of the continuity equation  \eqref{eq.pierre} and the ODE \eqref{takis41} in the form of  a Borel measure $\eta$ on $\Gamma=C([t,t_0], \R^d)$ which is concentrated on solutions of \eqref{takis41} and is such that $m(s)=e_s\sharp \eta$ for $s\in [t,T]$. 
\vskip.075in

We claim that,  for $\eta-$a.e. $\gamma\in \Gamma$, $\gamma$ is a Filippov solution of  \eqref{takis41}. Indeed, it follows from  \eqref{takis42} that,
since $m$ is absolutely continuous with respect to the Lebesgue measure, 
$$
0= \int_t^T \int_{\R^d} {\bf 1}_{\mathcal N}(x,s) m(x,s)dxds  =  \int_{\Gamma} \int_t^T{\bf 1}_{\mathcal N}(\gamma(s),s) ds\eta(d\gamma). 
$$
Hence,  for $\eta-$a.e. $\gamma\in \Gamma$,  $(\gamma_s,s)\notin \mathcal N$ for a.e. $s\in [t,T]$ and, therefore, for a.e. $s\in [t,T]$,
$$
\dot \gamma_s= b_t(\gamma_s) \in  b^F_s(\gamma_s).
$$
Thus, for $\eta-$a.e. $\gamma\in \Gamma$, $\gamma$ is a Filippov solution of \eqref{takis41}, and, in view of the  uniqueness of the backward solution, $\gamma(s)= X(\gamma(t_0),t_0,s)$. Since $\gamma(t)$ belongs to the support of $\bar m_0$ which is contained in $B_\delta(y)$, this leads to a contradiction. 
\end{proof}

\begin{lem}\label{remlem.1}  Assume \eqref{takis40}. Then, for any $s,t\in [0,T]$ with $s<t$, there exists a set $E_{s,t}$ of full $\mathcal L^d-$measure  on which $X^{-1}(\cdot, t,s)$ is a singleton.  
\end{lem} 

\begin{proof} The main step of the proof is the fact that, for  any $x\in \R^d$,  the set $X^{-1}(\{x\},s,t)$, which, in view of Lemma~\ref{lem.inject},  is nonempty,   is connected. 
\smallskip

Since $X^{-1}(\{x\},s,t)$ is compact, it suffices to show that, if  $O_1$ and $O_2$ are two open subsets of $\R^d$ such that $X^{-1}(\{x\},s,t) \subset O_1\cup O_2$ and  $\overline O_1\cap \overline O_2=\emptyset$, then $X^{-1}(\{x\},s,t)$ is contained either  in $O_1$ or in $O_2$. 
\smallskip

Let $O_1$ and $O_2$ be as above. 
The upper-semicontinuity of $X^{-1}(\cdot, s,t)$, which is a consequence of   the stability of the flow, yields some  $r>0$ such that $X^{-1}(B_r(x),s,t) \subset O_1\cup O_2$. 
\smallskip

Let $b^\ep$ and $X^\ep$ be respectively a smooth approximation of $b$ and the associated group of solution. Then  
$(X^\ep)^{-1}(B_r(x),s,t)= X^\ep(B_r(x),t,s)$ is connected and, for $\ep$ small enough, is contained in  $O_1\cup O_2$. Thus, it is contained either in $O_1$ or in $O_2$. Without loss of generality, we can assume that there exists $\ep_n\to0$ such that $(X^{\ep_n})^{-1}(B_r(x),s,t) \subset O_1$. 
\smallskip

Passing to the limit up to this subsequence, we infer that 
$$
X^{-1}(\{x\},s,t)  \subset \bigcap_{r>0} {\rm limsup} (X^{\ep_n})^{-1}(B_r(x),s,t) \subset \overline{O_1}. 
$$
Since $\overline O_1\cap \overline O_2=\emptyset$  and $X^{-1}(\{x\},s,t) \subset O_1\cup O_2$, this implies that $X^{-1}(\{x\},s,t)$ is contained in $O_1$, and, hence,   $X^{-1}(\{x\},s,t)$ is connected. 

\smallskip

The fact that  $x\to X(x,s,t)$ is Lipschitz continuous and the area formula imply that $X^{-1}(\{x\},s,t)$ is at most countable for a.e. $x\in \R^d$. Then the fact that $X^{-1}(\{x\},s,t)$  is nonempty and connected implies that, as soon it is countable,   $X^{-1}(\{x\},s,t)$ must be a singleton for a.e. $x\in \R^d$. 
\end{proof}

The next Lemma is about the existence and uniqueness of a forward solution of the Fillipov ODE.

\begin{lem}\label{lem.kensf} Assume \eqref{takis40}. Then,  there exists $E \subset R^d$of full $\mathcal L^d-$measure such that, for any $x\in E$, there exists a unique forward maximal absolutely continuous solution of  $\dot x_t\in b^F_t(x_t)$ on $[0,T]$ with $x_0=x$. 
\end{lem}

\begin{proof}  Let $(t_n)_{n\in N}$ be a countable  and dense set  of times in $[0,T]$ with $t_0=T$ and $E= \bigcap_n E_{0,t_n}$, where $E_{s,t}$ is given in Lemma \ref{remlem.1}.  Note that $E$ had a full  $\mathcal L^d-$measure in $\R^d$. Then,  for any $x\in E$, there exists a unique $y\in \R^d$ such that $X(y,T,0)=x$. 
\smallskip

We claim that $t\to X(y,T,t)$ is the claimed unique forward maximal solution. Indeed, by definition, it is a maximal solution. Assume that $x:[0,t^*)\to \R^d$ is another maximal solution defined on an interval $[0,t^*)$ with $t^*\in (0,T]$ and $x_0=x$. Note that, the backward uniqueness of the flow, implies that  $x_s= X(x_t, t,s)$ for any $0\leq s<t< t^*$. 
\smallskip

Let $t_n\in (0,t^*)$. Then $t\to X(y,T,t)$ is a solution on $[0,t_n]$ starting from $X(y,T,t_n)$. Therefore $x_{t_n}$ and $X(y,T,t_n)$ belong to $X^{-1}(\{x\}, t_n,0)$, which is a singleton by the definition of $E_{0,t_n}$. It follows that 
$x_{t_n}= X(y,T,t_n)$. So $x_t= X(y,T,t)$ on $[0,t^*)$, and, hence, the uniqueness.
\end{proof}

We are now in a position to complete the proof of the existence and uniqueness of bounded and absolutely continuous distributional  solutions of the continuity equation.

\begin{proof}[Proof of Proposition \ref{prop.appen}] Since the existence of a bounded solution of the continuity equation can be achieved by standard approximation, we   concentrate only on the uniqueness. 
\vskip.075in

Let $m$ be an absolutely continuous solution to the continuity equation with initial condition $\bar m_0$. 
The Ambrosio superposition Theorem \cite{Am04}  yields a measure $\eta$ on $\Gamma$ such that $m(t)= e_t\sharp \eta$, where $e_t(\gamma)=\gamma_t$,  and for $\eta-$a.e. $\gamma\in \Gamma$, $\gamma$ is an absolutely continuous  solution to the ODE $\dot \gamma_t= b_t(\gamma_t)$.  
\smallskip

Arguing as in  the proof of Lemma \ref{lem.inject}, we find that,  that  $\eta-$a.e. $\gamma\in \Gamma$ is a Filippov solution of  the ODE. 

\smallskip

We now disintegrate $\eta$ with respect to $\bar m_0$ into $\eta(d\gamma)=\int_{\R^d} \eta_x(d\gamma) \bar m_0(x)dx$ in such a way that, for $\bar m_0-$a.e. $x\in \R^d$ and $\eta_x-$a.e. $\gamma\in \Gamma$, $\gamma_0=x$. Since, by Lemma \ref{lem.kensf},  the Filippov solution to the ODE is unique for a.e. $x\in \R^d$, we obtain that,  for $\bar m_0-$a.e. $x\in \R^d$, $\eta_x$ is a  Dirac mass. The uniqueness of the bounded and absolutely continuous solution of the continuity equation then easily follows  from \cite{Am04}. 
\end{proof}

\bibliographystyle{siam}

\bigskip

\noindent ($^{1}$) Ceremade (UMR CNRS 7534) \\ Universit\'{e} Paris-Dauphine PSL\\ Place du Mar\'{e}chal De Lattre De Tassigny \\ 75775 Paris CEDEX 16, France\\
email: cardaliaguet@ceremade.dauphine.fr
\\ \\
\noindent ($^{2}$) Department of Mathematics, 
The University of Chicago, \\
5734 S. University Ave., 
Chicago, IL 60637, USA  \\ 
email: souganidis@math.uchicago.edu

\end{document}